\documentclass[11pt]{article}

\usepackage{amsfonts,mathtools, ntheorem}
\usepackage{amsmath,bbm,accents}
\usepackage[usenames,dvipsnames]{color}
\usepackage{amssymb}
\usepackage{graphicx,amsmath,calc}
\usepackage{hyperref}
\hypersetup{
    colorlinks=true,
    linkcolor=blue,
    filecolor=magenta,      
    urlcolor=cyan,
    citecolor=blue
}
\usepackage[font=scriptsize,labelfont=bf]{caption}

\usepackage{graphicx}

\makeatletter
\DeclareFontFamily{U}{tipa}{}
\DeclareFontShape{U}{tipa}{m}{n}{<->tipa10}{}
\newcommand{\arc@char}{{\usefont{U}{tipa}{m}{n}\symbol{62}}}%

\newcommand{\arc}[1]{\mathpalette\arc@arc{#1}}

\newcommand{\arc@arc}[2]{%
  \sbox0{$\m@th#1#2$}%
  \vbox{
    \hbox{\resizebox{\wd0}{\height}{\arc@char}}
    \nointerlineskip
    \box0
  }%
}

\newcommand{\beq}{\begin{equation}}
\newcommand{\eeq}{\end{equation}}

\newcommand{\bea}{\begin{eqnarray}}
\newcommand{\eea}{\end{eqnarray}}
\newcommand{\beas}{\begin{eqnarray*}}
\newcommand{\eeas}{\end{eqnarray*}}

\newtheorem{theorem}{Theorem}[section]

\newtheorem{definition}[theorem]{Definition}
\newtheorem{proposition}[theorem]{Proposition}

\newtheorem{lemma}[theorem]{Lemma}
\newtheorem{remark}[theorem]{Remark}
\newtheorem{foo}[theorem]{Remarks}

\newenvironment{Remark}{\begin{remark}\rm}{\end{remark}}

\theorembodyfont{\upshape}
\newtheorem{example}{Example}
\theoremstyle{empty}
\newtheorem{refproof}{Proof}
\newenvironment{proof}{\addvspace{\medskipamount}\par\noindent{\it
Proof}.}
{\unskip\nobreak\hfill$\Box$\par\addvspace{\medskipamount}}

\newcommand{\fh}{\mathfrak h}

\newcommand{\R}{\mathbb R}
\newcommand{\ep}{\epsilon}
\newcommand{\p}{\mathbb P}

\newcommand{\inte}{\mathbf Int}
\newcommand{\cl}{\mathbf Cl}
\newcommand{\h}{\mathbf{H}_1}
\newcommand{\Span}{\mathrm {Span}}

\newcommand{\Z}{\mathbb Z}

\newcommand{\fL}{\mathfrak L}

\newcommand{\fI}{\mathfrak I}

\newlength{\mywidth}

\title{Very rare events for diffusion processes in short time}
\author{G\'erard Ben Arous, Jing Wang}

\date{}

\begin{document}
\maketitle

\begin{abstract}
We study the large deviation estimates for the short time asymptotic behavior of a strongly degenerate diffusion process. Assuming a nilpotent structure of the Lie algebra generated by the driving vector fields, we obtain a graded large deviation principle and prove the existence of those ``very rare events". In particular the first grade coincides with the classical Large Deviation Principle. 
\end{abstract}

\section{Introduction}

Large deviations principles for diffusions in short time are very well understood, as well as the related study of short-time asymptotics for heat kernels, at least since the work of Schilder (\cite{Shielder}), Varadhan (\cite{var}), Freidlin-Wentzell (\cite{FW}), and Azencott (\cite{Azencott}). 
We study here a rather new aspect of this classical question, and show that for non-elliptic diffusions, these estimates may miss the right order of magnitude for certain events. 

Consider the solution of the Stratonovich stochastic differential equation

\begin{equation}\label{eq-sde}
dx(t)=\sum_{i=1}^mX_i(x(t))\circ dw_t^i+X_0(x(t))dt, \quad x(0)=x_0,
\end{equation}
where $X_0,\dots, X_m$ are smooth vector fields on a manifold $M$ and $w_t^i$, $i=1,\dots, m$ are independent  standard Brownian motions. This process is naturally the diffusion generated by the operator $L = \frac{1}{2} \sum_{i=1}^m X_i^2 + X_0$. (We skip in this preliminary discussion the natural and well-known assumptions needed for the existence and uniqueness of this process.)
For $\ep >0$ , define the rescaled process $x^{\ep}(t) = x({\ep^2 t})$, generated by $\ep^2 L$, the classical results mentioned above give a Large Deviation Principle (or LDP) for the distribution $\p^{\ep}$ of the rescaled process $x^{\ep}$ on the path space $E= C([0,1], M)$ (in this form, it is due to \cite{Azencott}, see also the reference books \cite{Dembo-Zeitouni} and \cite{DS}).

\begin{theorem}\label{Azencott}
The distribution $\p^{\ep}$ satisfies a Large Deviation Principle at rate $\ep^{-2}$, with rate function $I$.
For any Borel set $A \subset E=C([0,1],M)$, 

\begin{equation}
\liminf_{\ep\to0}\ep^{2}\log\p(x^\ep\in A)\ge -\inf( I(\phi), \phi \in  \ring{A})
\end{equation}
and
\begin{equation}\label{eq-main-nil-gr-2}
\limsup_{\ep\to0}\ep^{2}\log\p(x^\ep\in A)\le -\inf( I(\phi), \phi \in \overline{A}).
\end{equation}

\end{theorem}

In this Large Deviation Principle, the rate is $\ep^{-2}$ and the rate function $I$ on the path-space $C([0,1], M)$ is given by

\begin{equation}\label{ratefunction}
I(\phi) = \inf (\frac{1}{2} \|h\|_{\h}^2, \Phi_{x_0}(h) = \phi)
\end{equation}
where $\phi \in E$ and the functional $\Phi_{x_0}$ is defined on the Cameron Martin Space $\h$, by the following differential equation. For $x_0 \in  M$ and $h \in \h$ define $\phi =\Phi_{x_0}(h)$ to be the solution to 
\begin{equation}\label{ode}
d\phi(t) =\sum_{i=1}^mX_i(\phi(t))\circ dh_t^i, \quad \phi(0)=x_0.
\end{equation}

A remarkable fact about the rate function $I$ is that it does not depend at all on the drift term $X_0$ but only on the diffusion part, i.e. the vector fields $(X_1,\dots, X_m)$.
The sets of paths defined by the image of this map $\Phi_{x_0}(h)$ are usually called \emph{(finite energy) horizontal paths}.
So that the rate function $I(\phi)$ is finite if and only if $\phi$ is a horizontal path. It is thus clear that $\log \p^{\ep}(A)$  is of order at least $\ep^{-2}$ if the interior of $A$ contains a (finite energy) horizontal path. The order of magnitude of  $\log \p^{\ep}(A)$ might be much smaller if the closure of $A$ contains no such horizontal (finite energy) paths. We will call such sets \emph{non-horizontally accessible}. Our goal here is to begin the study of the order of magnitude for the probability of such sets which are very rare events. 

Let us first concentrate on the simplest and most natural  events $A$ of such kind, i.e. those that only depend on the final point of the diffusion path. Consider the distribution $Q^{\ep}$ of the end point of the diffusion, i.e. of $x^{\ep}(1) = x({\ep^2 })$.
The classical LDP given above obviously implies the following LDP for $Q^{\ep}$.

\begin{theorem}\label{Azencott2}
The distribution $Q^{\ep}$ satisfies a Large Deviation Principle at rate $\ep^{-2}$, with rate function $J$
where 
\begin{equation}
J(y) = \inf ( I(\phi), \phi(1) = y) = \inf \left( \frac{1}{2} \|h\|_{\h}^2, \Phi_{x_0}(1)=y \right).
\end{equation}
\end{theorem} 

It is well known that if the strong H\"ormander's condition is satisfied, i.e. if the Lie algebra generated by the vector fields $X_1,\dots, X_m$ is of full rank at every point $x_0 \in M$, then for any $y \in M$,  there is a finite energy horizontal path joining the starting point $x_0$ and $y$, so that the rate function $J(y)$ is finite for every $y \in M$. In fact $J(y)$ is then simply the sub-Riemannian (or Carnot-Carath\'eodory) distance between the initial point $x_0$ and $y$.
In this case the classical Large Deviation Principle given above provides the right order of magnitude for the probability $Q^{\ep} (B)$ for a Borel subset $B$ of $M$, if $B$ has a non empty interior.  But, if this Strong H\"ormander's condition is not true, very rare events may exist, where the classical Large Deviation Principle does not give the right order of magnitude, even when the weak H\"ormander's condition is satisfied, and thus even when a smooth heat kernel exists. This will naturally show that the classical logarithmic asymptotics for the heat kernel in short time cannot be valid, along the lines of the well known results due to Varadhan (\cite{var})  for the elliptic case, and L\'eandre (\cite{Leandre1}, \cite{Leandre2}) for the hypo-elliptic case, under the Strong H\"ormander's condition.

We begin with a very simple example, where the weak H\"ormander's condition is satisfied. This dispels the idea that these very rare events should be rather pathological.

\begin{example}\label{ex-kol}

Let us start here with the simplest possible example. 
Consider the so-called Kolmogorov diffusion $x(t)$ on $\R^2$ generated by the operator $L= \frac12 X_1^2 + X_0$, where the vector fields $X_0$ and $X_1$ are given by  $X_0=x^1 \frac {\partial}{\partial x^2}$ and $X_1=\frac{\partial}{\partial x^1}$. 
We fix here the initial condition $x(0)=0$. This diffusion is given by the solution of the stochastic differential equation
\begin{equation}
dx(t)=X_1(x(t))\circ dw_t+X_0(x(t))dt, \quad x(0)=0.
\end{equation}
Obviously the solution of this SDE is explicit and is given by the Gaussian process
\begin{equation}
x(t) = \left(w_t, \int_0^t w_sds\right).
\end{equation}

Consider the sets $B_1= \{ (x^1,x^2) \in \R^2, x^1 > 1 \}$ and $B_2= \{ (x^1,x^2) \in \R^2, x^2 > 1 \}$. 
Then obviously the LDP given above gives the right order of magnitude for the probability $\p ( x^{\ep}(1) \in B_1)$ but not for $\p ( x^{\ep}(1) \in B_2)$. Indeed $B_2$ is clearly an open and non-horizontally accessible set.  
The Large Deviation Principle given above only tells us that

\begin{equation}
\lim_{\ep\to0}\ep^{2}\log Q^{\ep}(B_2) = -\infty.
\end{equation}
We will see below that it is easy to compute a much better estimate for the probability $Q^{\ep}(B_2)$ of this very rare event. Indeed it is obvious here to compute the heat kernel, i.e. the density of this Gaussian process.
In this simple case, the right order of magnitude 
is given by
\begin{equation}\label{eq-kol-Q-B-2}
\lim_{\ep\to0}\ep^6\log Q^{\ep}(B_2) =\lim_{\ep\to0}\ep^6\log\p\left(\ep^3\int_0^1w_sds>1\right)=-\frac32.
\end{equation}

\end{example}

The goal of our paper is to show that such events exists in much more general contexts, and to study the order of magnitude of their probability. In this simple Example \ref{ex-kol},  we have (at least) two different powers of $\ep$ as rates for short time large deviations, i.e. $\ep^{-2}$ and $\ep^{-6}$, for different types of events.  We want to understand this phenomenon in greater generality.

Before discussing this generalization, it might be useful to discuss here a natural guess for a way to estimate the probability of these non-horizontally accessible events.  We first recall the classical Stroock-Varadhan support theorem, and then dwell more on our simple example. 

Define the functional $\Psi^{\ep}_{x_0}$ on the Cameron Martin Space $\h$ by the following differential equation. For $x_0 \in  M$ and $h \in \h$, let $\psi =\Psi^{\ep}_{x_0}(h)$ be the solution to 
\begin{equation}\label{odewithX0}
d\psi(t) =\ep \sum_{i=1}^mX_i(\psi(t))\circ dh_t^i + \ep^2 X_0(\psi(t))dt, \quad \psi(0)=x_0.
\end{equation}

The Stroock-Varadhan support theorem (\cite{SV}) says that the support of the distribution $\p^{\ep}$ is given by the closure in $E$ of the image 
$\Psi^{\ep}_{x_0}(\h)$. It is then tempting to guess that the order of magnitude of $\log \p^{\ep}(A)$ for an event $A \subset E$ is rather given by the infimum of $I^{\ep}$ than the infimum of $I$, where
\begin{equation}\label{ratewithdrift2}
I^{\ep}(\psi) = \inf \left(\frac{1}{2} \|h\|_{\h}^2, \Psi^{\ep}_{x_0}(h) = \psi\right).
\end{equation}

In the simple context of the Kolmogorov diffusion example above, it is indeed true that (cf Section \ref{sec-variation-kol})
\begin{equation}\label{claim1}
\lim_{\ep\to0} \frac{\log \p^{\ep}(A)}{\inf (I^{\ep}(\psi), \psi \in A)} =-1
\end{equation}
for both the sets $A_1= \{ \psi, \psi(1) \in B_1 \}$ and $A_2= \{ \psi, \psi(1) \in B_2 \}$.
But this fact is not always true. We will see that, still in the very simple case of the Kolmogorov diffusion, there exists a (rather pathological) set $A \subset E$, such that this guess is not correct (see Section \ref{sec-bad-set}). It would be interesting to characterize the sets for which this estimate is true. 

Our main result will not follow this route but rather use a very important characteristic of our simple example, which is that the Lie algebra generated by the vector fields driving the equation is nilpotent. Our main result generalizes this ``graded" behavior to the general case where the Lie algebra $\mathfrak{L}$ generated by the vector fields driving the equation \eqref{eq-sde} is nilpotent. 

We first introduce a simple definition.
\begin{definition} 
We call a Borel set $A \subset E$ to be of grade $\alpha$ for the probability measure $\p^{\ep}$ if
\[
-\infty<\liminf_{\ep\to0}\ep^{2\alpha}\log\p^{\ep}(A)\le\limsup_{\ep\to0}\ep^{2\alpha}\log\p^{\ep}(A)<0.
\]

\end{definition}

We now give our first general result.

\begin{theorem}\label{mainresult1}
Assume that the vector fields $X_0,\dots,X_m$ are complete and generate a nilpotent Lie algebra $\mathfrak{L}$.
There exist positive (rational) numbers $1=\alpha_1<\cdots<\alpha_\ell<\infty$, such that for each $1\le k\le\ell$, there exist Borel subsets $A_k$ of grade $\alpha_k$.

\end{theorem}

We will prove this in Section \ref{sec-general}. In fact we will prove there our main result Theorem \ref{mainresult3}, which is sharper and more quantitative, since it gives a sufficient condition to check if an event is of grade $\alpha_k$. This condition is algebraic in nature and rather complex. It will imply in particular that

\begin{theorem}\label{mainresult2}
Under the assumptions of Theorem \ref{mainresult1}, the sets $A_k$ can be chosen to depend only on the final point, i.e. there exists  sets $B_k \subset M$ such that $A_k = \{ \phi \in E, \phi(1) \in B_k \}$.

\end{theorem}

The nilpotence assumption may seem too restrictive, and it probably is, for understanding the general phenomenon of ``very rare events". But this assumption is important to get the result above about the different grades being powers of $\ep$. Indeed it is easy to see that, even though very rare events do exist, this grading behavior will be not be valid if we consider even the simplest case where the Lie algebra  $\mathfrak{L}$ is not nilpotent but solvable.

\begin{example}\label{ex-solvable}
We look here at a diffusion which can be seen as the solution of the following SDE on $\R^2$
\begin{equation}\label{eq-sde-solv}
dx(t)=X_1(x(t))\circ dw_t+X_0(x(t))dt, \quad x(0)=0,
\end{equation}
where the vector fields $X_0$ and $X_1$ are given by  $X_0=e^{x^1} \frac {\partial}{\partial x^2}$ and $X_1=\frac{\partial}{\partial x^1}$.  The solution of this SDE is also explicit and is given by
\begin{equation}
x(t)=\left(w_t, \int_0^t e^{w_s}ds\right).
\end{equation}
Again, for simplicity let us consider the same event $B_2= \{ (x^1,x^2) \in \R^2, x^2 > 1 \}$. Again here, the event is not horizontally accessible. The computation is a bit more delicate but it is still possible to compute the right order of magnitude of the probability of event (see Section \ref{sec-solv}). 
\begin{equation}
\lim_{\ep\to0}\frac{\ep^2}{\log^2(1/\ep)}\log \p( x^{\ep}(1) \in B_2) =-2
\end{equation}
Here we also have two rates, but they are not polynomial in $\ep$ since one of them is logarithmic. This comes from the fact that the Lie algebra $\mathfrak{L}$  is solvable but not nilpotent. It is natural to assume that a similar lines of results is true for general solvable diffusions.

\end{example}

Assuming nilpotence  of $\fL$, our strategy is to exploit a well known tool for the diffusions $x(t)$, known as the stochastic Taylor formula (see Yamato \cite{YAM}, Castell \cite{CAS}, \cite{GBA}), giving an explicit representation. It shows that the diffusion $x(t)$ can be seen as a smooth and explicit function of the family of all Stratonovich iterated integrals of length at most $r$. 
More precisely, for any integer $k$, and multi-index $J$ of length $|J|=k$ in $\{0,\dots,m\}^k$, say $J = (j_1,\dots, j_k)$, define the iterated Stratonovich integral
\begin{equation}
W^J_t =  \int_{0< t_1< \cdots < t_k<t} dw^{j_1}_{t_1}\circ \cdots \circ dw^{j_k}_{t_k},
\end{equation}
where we use the convention that $w^0_t = t$.
Consider the family of all such iterated integrals of length less than or equal to $r$, and denote it by $Y_t = (W^J_t)_{|J| \leq r}$.
The stochastic Taylor formula shows that there exists a smooth function $F$ such that
\begin{equation}
x(t)= F(x_0, Y_t).
\end{equation}
In fact this function $F$ is very explicit, see \eqref{eq-taylor-F}.  At last in Section \ref{sec-example}, we will focus on the simple examples that are mentioned above, give explicit computations, and compare to the estimates obtained by applying Theorem \ref{mainresult3}.

\section{Graded Large Deviations for a universal nilpotent diffusion}\label{sec-Y}

In this section we consider the ``universal" nilpotent diffusion $Y_t = (Y^J_t)_{|J| \leq r}$ where $Y^J_t:=W^J_t$ on $\R^D$. It is known that $Y_t$ is in fact a solution of a SDE on $\R^D$ (see \cite{YAM}),
\[
dY_t=\sum_{i=0}^m Q_i(Y_t) \circ dw^i_t,\quad |J|\le r,
\]
where in Cartesian coordinates $(y^J, |J|\le r)$ of $\R^D$ we have explicitly $Q^J_i$, $J=(j_1,\dots, j_k)$ given by
\[
Q^{(j_1,\dots, j_k)}_i=
\begin{cases}
& y^{(j_1,\dots, j_{k-1})}\delta^{j_k}_{i}\frac{\partial}{\partial y^J}, \quad k>1 \\
& \delta^{j_k}_i\frac{\partial}{\partial y^J}, \quad k=1
\end{cases},
\]
where $\delta^{j}_i$ denotes the Kronecker delta function of integers $i, j$.
\begin{Remark} In fact it is a bit too large a ``universal" diffusion. It would be more natural to work with the natural diffusion on the free nilpotent algebra with $m+1$ generators and of step $r$, but this would complicate a bit our exposition.
\end{Remark}

We will need a few simple definitions before we can introduce our first graded large deviation theorem for the universal nilpotent diffusion $Y_t$.

\begin{definition}
For any multi-index $J = (j_1,\dots, j_k)$, we denote by $p(J)$ the number of zeros, and $n(J)$ the number of non-zeros in $J=(j_1,\dots, j_k)$. 
We call $\|J\|:=n(J)+2p(J)$ the \emph{size} of $J$. 
When $n(J)$ is not zero,  the $\alpha$-index of $J$ is defined to be
\[
\alpha(J):=1+\frac{2p(J)}{n(J)}=\frac{\|J\|}{n(J)}.
\] 

\end{definition}

We use this definition of the $\alpha$-indices to define a flag of vector subspaces of $\R^D$ corresponding to the $\alpha$-indices.  We first extend naturally the definition by defining $\alpha(J)$ to be infinite if $n(J)=0$.
Denote by $\{e_J, |J|\le r\}$ the basis of  $\R^D$. For any $\alpha>0$, let
 \begin{equation}\label{eq-W-alpha}
 \tilde{W}(\alpha)=\Span \{e_J,\alpha(J)\le\alpha \}.
 \end{equation}
 Clearly $  \tilde{W}(\alpha')\subset   \tilde{W}(\alpha)$ if $\alpha'\le\alpha$. Let
 \begin{equation}\label{eq-d-alpha}
 d(\alpha)=\dim   \tilde{W}(\alpha),
 \end{equation}
then $d:[1,+\infty)\to \mathbb{Z}^+$  is a right continuous increasing step function and $0\le d(\alpha)\le D$. Let $ {\alpha}_j$, $1\le j\le \jmath$ be such that $\lim_{\delta\to 0+}d( {\alpha}_j-\delta)<d( {\alpha}_j)$. 
Then we obtain a \emph{graded structure} $1= {\alpha}_1<\cdots < {\alpha}_\jmath<\infty$ and the corresponding flag 
\begin{align}\label{eq-G-seq}
\tilde{W}({\alpha}_1)\subsetneq\cdots\subsetneq \tilde{W}({\alpha}_{\jmath})\subsetneq \R^D.
\end{align}
We denote by $\Pi^{{\alpha}_k}$ is the natural projection map from $\R^D$ to $\tilde{W}({\alpha}_k)$.

We then define an important family of dilations on $\R^D$.
For any $1\le k\le \jmath$ and any multi-index $J$, with $|J|\le r$, we define

\begin{equation}\label{gamma-index}
\gamma^{{\alpha}_k}(J)=({\alpha}_k\cdot n(J)-\|J\|)_+=(n(J)({\alpha}_k-\alpha(J)))_+.
\end{equation}

\begin{definition}
For any $\eta>0$ , and $1\le k\le \jmath$, we define the dilation map $T_\eta^{{\alpha}_k}$ on $\R^D$ by
 \begin{equation}
T^{{\alpha}_k}_\eta(v)=(\eta^{\gamma^{{\alpha}_k}(J)}v_J)_{|J|\le r},
 \end{equation}
for $v\in \R^D$.
\end{definition} 

Our first result is a Large Deviation Principle for the dilated processes $Y^{\ep, k}:=T_\ep^{{\alpha}_k}(Y^\ep)$, for each 
$1\le k\le \jmath$.

\begin{theorem}\label{thm-ldp-Y}
For each $1\le k\le \jmath$, the distribution of dilated process $Y^{\ep, k}:=T_\ep^{{\alpha}_k}(Y^\ep)$ satisfies a Large Deviation Principle at rate $\ep^{-2{\alpha}_k}$ with rate function 
\begin{equation}\label{eq-rate-Y-k}
I^{{\alpha}_k}(\varphi):=\inf\left(\frac12\|h\|^2_{\h}, \Phi_0^{{\alpha}_k}(h)=\varphi\right),
\end{equation}
where $\Phi_0^{{\alpha}_k}(h):=\Pi^{{\alpha}_k}\circ {\Phi}^Y_0(h)$. $\Pi^{{\alpha}_k}$ is the projection map from $\R^D$ to $\tilde{W}({\alpha}_k)$, and  $\Phi_0^Y(h)=:\phi$ is the solution of the ODE on $\R^D$:
\[
d\phi_t=\sum_{i=1}^m Q_i(\phi_t) dh^i_t, 
\quad \phi_0=0.
\]
\end{theorem}

\begin{proof}
We begin by a simple but important scaling argument, for which we will need a more flexible notation here for stochastic iterated integrals, singling out the role of the Brownian stochastic integrals versus the deterministic ones. 
For any multi-index $J = (j_1,\dots, j_k) $, we will denote by
\begin{equation}
I_J(W_t, t)= Y^J_t= W^J_t =  \int_{0< t_1< \cdots < t_k<t} dw^{j_1}_{t_1}\circ \cdots \circ dw^{j_k}_{t_k},
\end{equation} 
where we use again the convention that $w^0_t = t$.

The scaling argument at the core of our argument is given in the following obvious lemma. 

\begin{lemma}
For any multi-index $J = (j_1,\dots, j_k) $, the three processes $Y^J_{\ep^2t}$, $\ep^{\|J\|} Y^J_t$ and $I_J(\ep^{\alpha(J)}W_t, t)$ have the same distribution.
\end{lemma}

\begin{proof}
If we rescale time by a factor $\ep^2$, the natural Brownian scaling invariance shows that $Y^J_{\ep^2t}$  has the same distribution  as $\ep^{\|J\|} Y^J_t$.
Indeed $Y^J_{\ep^2t}  =I_J(W_{\ep^2t}, \ep^2t)$ has the same distribution as $  \ep^{\|J\|} I_J(W_t,t)=  \ep^{\|J\|}  Y^J_t$. It is also easy to see that the distribution of $Y^J_{\ep^2 t}$ is the same as the distribution of the stochastic integral $I_J(\ep^{\alpha(J)}W_t, t)$, where we rescale the Brownian integrands by $\ep^{\alpha(J)}$ but do not rescale the time integrands (i.e. $dw^0_t$). Indeed $I_J(\ep^{\alpha(J)}W_t, t)= \ep^{\alpha(J) n(J)} I_J(W_t, t)= \ep^{\|J\|} Y^J(t)$. 
\end{proof}
 
We consider now the process

\[
(Z^{\ep, k})^J=
\begin{cases}
&\ep^{{\alpha}_k\cdot n(J)}W^J_t, \quad \alpha(J)\le {\alpha}_k\\
&0, \qquad \qquad\quad\ \alpha(J)> {\alpha}_k
\end{cases},
\quad |J|\le r.
\]

\begin{lemma}
The process $Z^{\ep, k}$ satisfies a Large Deviation Principle at rate $\ep^{-2{\alpha}_k}$ and rate function
\begin{equation}
I^{{\alpha}_k}(\varphi):=\inf\left(\frac12\|h\|^2_{\h}, \Phi_0^{{\alpha}_k}(h)=\varphi\right),
\end{equation}
where $\Phi_0^{{\alpha}_k}(h):=\Pi^{{\alpha}_k}\circ {\Phi}^Y_0(h)$.
\end{lemma}

\begin{proof}
Consider the process $z^{\ep}$ in $C([0,1],\R^D)$, defined as the solution to the SDE
\[
dz^\ep_t=\ep \sum_{i=1}^m Q_i(z^\ep_t) \circ dw^i_t+ Q_0(z^\ep_t) dt.
\]
With our notations above, we have $z^{\ep}=(z^\ep_J)_{|J|\le r}$ where
\begin{equation}
z_J^{\ep}(t) = I_J( \ep W_t, t).
\end{equation}
So that $z_J^{\ep^{\alpha_k}}(t)$ has the same distribution as $\ep^{\alpha_k n(J)} I_J(W_t, t)$.
This proves that the process $Z^{\ep, k}$ is given by a simple projection:
\begin{equation}
Z^{\ep, k} = \Pi^{\alpha_k}(z^{\ep^{\alpha_k}}_t).
\end{equation}
This allows us  to prove very simply a Large Deviation Principle for $Z^{\ep, k}$.
Indeed, we know that (see \cite{Azencott}), the distribution of the process $z^{\ep}$ satisfies a Large Deviation Principle with rate $\ep^{-2}$ and rate function
\begin{equation}
I_0(\varphi):=\inf\left(\frac12\|h\|^2_{\h}, \Phi^Y_0 (h)=\varphi\right),
\end{equation}
where  $\Phi^Y_0 (h)=\phi$ is the solution of the ODE in $C([0,1],\R^D)$:
\[
d\phi(t)=\sum_{i=1}^m Q_i(\phi(t)) dh^i_t, 
 \quad \phi(0)=0.
\]
A simple contraction principle and the Large Deviation Principle for  $z^\ep$ show that $Z^{\ep, k}$ satisfy a Large Deviation Principle at rate $\ep^{-2{\alpha}_k}$ and rate function
\[
I^{{\alpha}_k}(\varphi):=\inf\left(\frac12\|h\|^2_{\h}, \Phi_0^{{\alpha}_k}(h)=\varphi\right),
\]
where $\Phi_0^{{\alpha}_k}(h):=\Pi^{{\alpha}_k}\circ {\Phi}^Y_0(h)$.
\end{proof}

We now prove the theorem. From the construction of the $\gamma^{{\alpha}_k}$-indices \eqref{gamma-index} we easily see that the $J$-th component of
$Y^{\ep, k}$ is given by
\[
(Y^{\ep, k})^J=
\begin{cases}
&\ep^{{\alpha}_k\cdot n(J)}W^J_t, \quad \alpha(J)\le {\alpha}_k\\
&\ep^{\alpha(J)\cdot n(J)}W^J_t, \quad \alpha(J)> {\alpha}_k.
\end{cases}
\]
We thus will use $Z^{\ep, k}$ to approximate $Y^{\ep, k}$, for each $1\le k\le \jmath$.

\begin{lemma}\label{lemma-z}
The process $Z^{\ep,k}$ is an ${\alpha}_k$- exponentially good approximation of $Y^{\ep,k}$, i.e. 
\begin{equation}
\limsup_{\ep\to0}\ep^{2\alpha_k}\log \p(\|Z^{\ep,k}- Y^{\ep,k}\|_{[0,1], \infty} >\delta)=-\infty.
\end{equation}
\end{lemma}

\begin{proof}
We first note that the Large Deviation Principle for the process $z_J^{\ep}(t) = I_J( \ep W_t, t)$  shows that there exists a constant $C_\delta > 0 $ such that
\[
\limsup_{\ep\to0}\ep^{2\alpha(J)}\log\p(\| I_J( \ep^{\alpha(J)} W_t, t) \|_{[0,1], \infty} >\delta)\le -C_\delta.
\]
So that, for any $\alpha < \alpha(J)$
\[
\limsup_{\ep\to0}\ep^{2\alpha}\log \p(\| I_J( \ep^{\alpha(J)} W_t, t) \|_{[0,1], \infty}>\delta) = - \infty.
\]
We now remark that 
\begin{align*}
Y^{\ep,k}-Z^{\ep,k}=(\ep^{\alpha(J)\cdot n(J)}W^J_t)_{ \alpha(J)> {\alpha}_k, |J|\le r},
\end{align*} 
which proves that
\begin{equation}
\limsup_{\ep\to0}\ep^{2\alpha_k}\log \p(\|Z^{\ep,k}- Y^{\ep,k}\|_{[0,1], \infty} >\delta)=-\infty,
\end{equation}
and concludes the proof of the lemma. 
\end{proof}
The fact that $Z^{\ep,k}$ satisfies a Large Deviation Principle at rate $\ep^{-2\alpha_k}$ and this approximation result show that $Y^{\ep,k}$ satisfies a Large Deviation Principle with the same rate and same rate function, which closes the proof of the theorem.
\end{proof}

\begin{Remark}
Note that the statement of Theorem \ref{thm-ldp-Y} for the grade $\alpha_1$ is a direct consequence of the classical LDP (Theorem \ref{Azencott}) applied to $Y^\ep$. Indeed $Y^{\ep,1}=\Pi^{{\alpha}_1}(Y^\ep)$ and the contraction principle together with Theorem \ref{Azencott} imply that $Y^{\ep,1}$ satisfies a LDP with rate $\ep^2$, and rate function
$I^{{\alpha}_1}(\varphi):=\inf(I(\psi),\Pi^{{\alpha}_1}(\psi)=\varphi)$. 

It is easy to see that $I^{{\alpha}_1}(\varphi):=\inf(\frac12\|h\|^2_{\h},\Pi^{{\alpha}_1}\circ {\Phi}^Y_0(h)=\varphi)$, which is exactly the rate function of Theorem \ref{thm-ldp-Y} for $\alpha_1=1$, since  $T^{{\alpha}_1}_\eta=\mathrm{Id}$ and $Y^{\ep,1}=\Pi^{{\alpha}_1}(Y^\ep)$.
\end{Remark}

We are now ready to give large deviation estimates for the distribution of $Y^\ep$ itself. In order to state our result we need to introduce new notions of \emph{closed (open) $\alpha$-dilation of $A$} for any Borel set $A\subset{C([0,1],\R^D)}$.

\begin{definition}
For any measurable sets $A\subset C([0,1],\R^D)$ (or $\R^D$) and $\alpha>0$, we call
\begin{equation}\label{eq-univ-cl-inte}
\cl^{\alpha}(A)=\cap_{\delta>0}\overline{\cup_{\eta\le\delta}\,T_\eta^{\alpha}(A)}, \quad 
\mbox{and}\quad
\inte^{\alpha}(A)=\cup_{\delta>0}\accentset{\circ}{\arc{\cap_{\eta<\delta}\,T^{\alpha}_{\eta}(A)}}
\end{equation}
the closed $\alpha$-dilation of $A$ and the open $\alpha$-dilation of $A$.
\end{definition}

We now state our graded large deviation estimates for $Y^\ep$.
\begin{theorem}\label{thm-ldp-Y-main}
There exists an integer $\jmath \ge 1$, rational numbers $1={\alpha}_1<\cdots<{\alpha}_\jmath$, such that 
for any Borel set $A \subset C([0,1],\R^D)$, we have

\begin{equation}\label{eq-univ-gr-1}
\liminf_{\ep\to0}\ep^{2{{\alpha}_k}}\log\p(Y^\ep\in A)\ge -\inf\left(I^{{\alpha}_k}(\varphi),\varphi\in \inte^{{\alpha}_k}(A)\right)
\end{equation}
and
\begin{equation}\label{eq-univ-gr-2}
\limsup_{\ep\to0}\ep^{2{{\alpha}_k}}\log\p(Y^\ep\in A)\le -\inf\left(I^{{\alpha}_k}(\varphi),\varphi\in \cl^{{\alpha}_k}(A)\right).
\end{equation}
\end{theorem}
\begin{Remark}
Note that this statement is not a LDP for $k\ge 2$. The rate functions $I^{{\alpha}_k}$ are good, but the sets $\cl^{{\alpha}_k}(A)$ and $\inte^{{\alpha}_k}(A)$ are not the closure nor the interior of $A$ in a topological sense. An alternative expression for the ``rate functions" \eqref{eq-univ-gr-1} and \eqref{eq-univ-gr-2} can be given by
\begin{equation}\label{eq-univ-gr-3}
\liminf_{\ep\to0}\ep^{2{{\alpha}_k}}\log\p(Y^\ep\in A)\ge -\inf\left(\frac12\|h\|^2_{\h},\Pi^{{\alpha}_k}\circ {\Phi}^Y_0(h)\in  \inte^{{\alpha}_k}(A)\right)
\end{equation}
and
\begin{equation}\label{eq-univ-gr-4}
\limsup_{\ep\to0}\ep^{2{{\alpha}_k}}\log\p(Y^\ep\in A)\le -\inf\left(\frac12\|h\|^2_{\h},\Pi^{{\alpha}_k}\circ {\Phi}^Y_0(h)\in \cl^{{\alpha}_k}(A)\right).
\end{equation}
\end{Remark}
\begin{refproof}[Proof of Theorem \ref{thm-ldp-Y-main}]
(1) We first prove the upper bound  \eqref{eq-univ-gr-1}. We just need to show that
\[
\limsup_{\ep \to0}\ep^{2{\alpha}_k}\log\p\bigg(Y^{\ep,k} \in T_\ep^{{\alpha}_k} \left(A \right)\bigg)\le -\inf\left(I^{{\alpha}_k}(\varphi),\varphi\in \cl^{{\alpha}_k}(A)\right).
\]
For any $\delta>0$, when $\ep$ is small enough we have
\begin{align*}
\p\left(Y^{\ep,k} \in T_\ep^{{\alpha}_k}(A) \right)
\le
\p \left(Y^{\ep,k}\in \overline{\cup_{\eta\le\delta}\,T^{{\alpha}_k}_{\eta}(A)}\right).
\end{align*}
Apply Theorem \ref{thm-ldp-Y} we obtain that
\[
\limsup_{\ep\to0} \ep^{2{\alpha}_k}\log\p\left(Y^{\ep,k}\in \overline{\cup_{\eta\le\delta}\,T^{{\alpha}_k}_{\eta}(A)}\right)
\le
-\inf \left(I^{{\alpha}_k}(\varphi), \varphi \in \overline{\cup_{\eta\le\delta}\,T^{{\alpha}_k}_{\eta}(A)}\right):=-m_\delta.
\]
Thus for all $\delta>0$, $\limsup_{\ep \to0}\ep^{2{\alpha}_k}\log\p\bigg(Y^{\ep,k}\in T_\ep^{{\alpha}_k}(A)\bigg)\le -m_\delta$.
 Let $m'=\sup_{\delta>0} m_\delta$, then
\[
\limsup_{\ep\to0}\ep^{2{\alpha}_k}\log\bigg(Y^{\ep,k}\in T_\ep^{{\alpha}_k}(A)\bigg)
\le
-m'.
\]
Denote $m=\inf\left(I^{{\alpha}_k}(\varphi),\varphi\in \cl^{{\alpha}_k}(A)\right)$ where $\cl^{{\alpha}_k}(\cdot)$ is as in \eqref{eq-univ-cl-inte}. We claim that $m=m'$. It is straight forward to see $m'\le m$ by noting $\cap_{\delta}\overline{\cup_{\eta\le\delta}\,T^{{\alpha}_k}_{\eta}(A)}\subset \overline{\cup_{\eta\le\delta}\,T^{{\alpha}_k}_{\eta}(A)}$.
To prove $m'\ge m$, we consider the following situations: 
\begin{itemize}
\item[(a)] If $m'=+\infty$, since $m'\le m$ we have $m=+\infty=m'$.
\item[(b)] If $m'<+\infty$. That is to say for any $\delta>0$, $m_\delta\le m'<+\infty$. Since $\overline{\cup_{\eta\le\delta}\,T^{{\alpha}_k}_{\eta}(A)}$ is closed and $I^{{\alpha}_k}$ is lower semicontinuous,  there exists  an $\varphi_\delta\in C([0,1], \R^D)$ such that 
\[
I^{{\alpha}_k}(\varphi_\delta) =m_\delta.
\]
We consider a sequence $\delta_n\to0$ as $n\to0$, then for any fixed $\delta>0$, we have for $n$ large enough that 
\[
\varphi_{\delta_n}\in \overline{\cup_{\eta\le\delta_n}\,T^{{\alpha}_k}_{\eta}(A)} 
\subset
\overline{\cup_{\eta\le\delta}\,T^{{\alpha}_k}_{\eta}(A)}.
\]
Hence there exists a $\varphi_0\in C([0,1],\R^D)$ such that $\varphi_0\in \overline{\cup_{\eta\le\delta}\, T^{{\alpha}_k}_{\eta}(A)}$ and
\[
\lim_{n\to0}\|\varphi_0-\varphi_{\delta_n}\|_{[0,1],\infty}=0,
\]
and satisfies that $I^{{\alpha}_k}(\varphi_0) \le m'$. At the end note $\varphi_0\in \cap_\delta \overline{\cup_{\eta\le\delta}\,T^{{\alpha}_k}_{\eta}(A)}=\cl^{{\alpha}_k}(A)$, we obtain
\[
m\le I^{{\alpha}_k}(\varphi_0)\le m',
\] 
hence $m=m'$  and the conclusion.
\end{itemize}

(2) The proof of lower bound  \eqref{eq-univ-gr-1} is similar. For any fixed $\delta>0$ and $\ep$ small enough we have
\begin{align*}
\p\left(Y^{\ep,k} \in T^{{\alpha}_k}_{\ep}(A)\right)\ge
 \p \left(Y^{\ep,k} \in
\accentset{\circ}{\arc{\cap_{\eta<\delta}\,T^{{\alpha}_k}_{\eta}(A)}}
\right).
\end{align*}
Apply Theorem \ref{thm-ldp-Y} we obtain that
\[
\liminf_{\ep\to0}\ep^{2\alpha_k}\log \p\left(Y^{\ep,k} \in T^{{\alpha}_k}_{\ep}(A)\right)
\ge
-\inf \left(I^{{\alpha}_k}(\varphi), \varphi \in \accentset{\circ}{\arc{\cap_{\eta<\delta}\,T^{{\alpha}_k}_{\eta}(A)}} \right):=-m_\delta.
\]
We denote $m=\inf\left(I^{{\alpha}_k}(\varphi), \varphi \in \inte^{{\alpha}_k}(A)\right)$ and $m'=\inf_{\delta>0} m_\delta$, we claim that $m=m'$.
Obviously we have $m'\ge m$. 
To show $m'\le m$ we consider the following two cases.
\begin{itemize}
\item[(a)] If $m=+\infty$, then $m'=+\infty$ hence $m'=m$.
\item[(b)] If $m<+\infty$, then for any $\delta'>0$, there exists an $\varphi \in \inte^{{\alpha}_k}(A)$ such that 
$I^{{\alpha}_k}(\varphi)\le m+\delta'$.
However since $\inte^{{\alpha}_k}(A)=\cup_{\delta>0}\accentset{\circ}{\arc{\cap_{\eta<\delta}\,T^{{\alpha}_k}_{\eta}(A)}}$, so $ \varphi\in \accentset{\circ}{\arc{\cap_{\eta<\delta}\,T^{{\alpha}_k}_{\eta}(A)}}$ for some $\delta_0>0$. Thus
\[
I^{{\alpha}_k}(\varphi)\ge m_{\delta_0}\ge m'.
\]
Therefore we have $m+\delta'\ge m'$ for any $\delta'>0$, hence $m\ge m'$ and the conclusion.
\end{itemize}
The proof  is then completed.
\end{refproof}

\section{General nilpotent case}\label{sec-general}

In this section we introduce graded large deviation estimates for the diffusion process $x(t)$,

\begin{equation}\label{eq-sde}
dx(t)=\sum_{i=1}^mX_i(x(t))\circ dw_t^i+X_0(x(t))dt, \quad x(0)=x_0,
\end{equation}
where the driving vector fields $X_0,\dots, X_m$ generate a nilpotent Lie algebra $\mathfrak{L}$.  

We begin by recalling here the main tool of our approach, i.e. the stochastic Taylor formula (see Yamato \cite{YAM}, Castell \cite{CAS}, Ben Arous \cite{GBA}). In the nilpotent case, this formula is given by 
\begin{equation}\label{eq-taylor-F}
x(t)= F(x_0, Y_t)=\exp_{x_0}\left(\sum_{|J|\le r}c^J(W_t,t)X^J\right),\quad 
c^J(W_t,t)=\sum_{\sigma\in\sigma_{|J|}}\frac{(-1)^{e(\sigma)}}{|J|^2
{|J|-1\choose e(\sigma)}}W^{J\circ \sigma^{-1}}.
\end{equation}
and it connects $x(t)$ to the ``universal" nilpotent diffusion $Y_t$.
The key step here is to describe the contraction from $\R^D$ to the ideal $\fI=\{X^J, J\not=0\}$ of $\fL$. 
Then we will obtain graded large deviation estimates for Borel sets in $C([0,1],\fL)$, which induce graded large deviation estimates for Borel sets in $C([0,1],M)$

As before, in order to introduce a natural dilation strategy, we need to construct the following flags of $\fL$. 
\subsubsection*{First flag with respect to the $\alpha$-grading}
 For any $\alpha>0$, consider $W_\alpha$ the vector space generated in the (finite dimensional) Lie algebra $\fL$ by the brackets $X^J$ with $\alpha(J)\le \alpha$, i.e.
 \begin{equation}\label{eq-W-alpha}
 W(\alpha)=\Span \{X^J,\alpha(J)\le\alpha \}.
 \end{equation}
 It then induces  a {graded structure} $1=\alpha_1<\cdots <\alpha_\ell<\infty$ with a flag 
 \begin{align}\label{eq-G-seq}
W(\alpha_1)\subsetneq\cdots\subsetneq W(\alpha_\ell)=\fI\subseteq \fL,
\end{align}
where $\alpha_j$, $1\le j\le \ell$ are such that $\lim_{\delta\to 0+}\dim W(\alpha_j-\delta)<\dim W(\alpha_j)$.
Note $W(\alpha_\ell)=\fI\subsetneq \fL$ if and only if $X_0\not\in W(\alpha_\ell)$. We have $ \fL=W(\alpha_\ell)\oplus \Span\{X_0\}$.

\subsubsection*{Secondary flag with respect to the dilation strength for each $\alpha_k$}
For each fixed $1\le k\le\ell$, let $\gamma^k(\cdot)$ be such that 
\[
\gamma^k(J)=n(J)(\alpha_k-\alpha(J)), \quad \forall\, |J|\le r, X^J\not=0.
\]
In particular $\gamma^k((0))=-\infty$. For any $\gamma\ge0$, we consider the space 
 \begin{equation}\label{eq-W-alpha}
 V^k(\gamma)=\Span \{X^J,\gamma^k(J)\ge\gamma 
 \}.
 \end{equation}
Similarly as before, it induces a sequence $\gamma^k_1>\cdots>\gamma^k_{\ell_k}=0$ and a corresponding flag of $W(\alpha_k)$:
\begin{align}\label{eq-G-seq}
V^k_{1}\subsetneq\cdots\subsetneq V^k_{\ell_k}= W(\alpha_k),
\end{align}
where $V^k_j= V^k(\gamma^k_{j})$, $j=1,\dots, \ell_k$. 

Let  $\mathcal{B}_j$ be the collection of words $|J|\le r$ such that $n(J)(\alpha_k-\alpha(J))=\gamma^k_j$, by definition of $\gamma^k_j$, $\{X^J, J\in \mathcal{B}_j\}$ generates new dimensions in $V^k_j$ that are not in $V^k_{j-1}$, i.e.,
\[
V^k_j=V^k_{j-1}\oplus\Span\{X^J, J\in \mathcal{B}_j\}.
\]
We can then define maps $\Psi^k_j: \R^{D}\to V^k_j $, $1\le j\le \ell_k$ such that for any $v=(v^J)\in\R^D$,
\[
\Psi^k_j((v^J)_{|J|\le r})=\sum_{K\in \mathcal{B}_j}v^K X^K.
\]

We call the vector $\gamma^k:=(\gamma^k_1,\dots, \gamma^k_{\ell_k})$ the dilation strength at grade $\alpha_k$. However, in order to introduce a dilation map on $\fL$, we need to decompose it into direct sums. Of course such a decomposition is not intrinsic. It is necessary to introduce the following notion of block structure.

\subsubsection*{Block structure and dilations for $\alpha_k$ grade}
%
 
Consider a block structure $U^k=(U^k_j)_{j=1}^{\ell_k}$ of $W(\alpha_k)$ which is adapted to the flag \eqref{eq-G-seq}, and let $U^k_{\ell_k+1}$ be such that $W(\alpha_k)\oplus U^k_{\ell_k+1}=\fI$. 
\[
\fI=U^k_1\oplus\cdots\oplus U^k_{\ell_k}\oplus U^k_{\ell_k+1}.
\]
Of course $\fL=U^k_1\oplus\cdots\oplus U^k_{\ell_k}\oplus U^k_{\ell_k+1}\oplus U^k_0$, where
\[
U^k_0=\begin{cases}
&\Span\{X_0\},\quad \mbox{if $\fI\subsetneq\fL$} \\
&0 \quad \mbox{if $\fI=\fL$} 
\end{cases}.
\]

Denote the projection map on  $U^k_j$ by $\Pi_j^{k}$, $j=1,\dots, \ell_k+1$.
We define another map $\Phi^k: \R^D\to \fI$ such that each component $\Phi_j^{k}: \R^{D}\to U^k_j$, $j=1,\dots, \ell_k$ is given by
\[
\Phi^{k}_j(v)= \Pi^{k}_j\circ \Psi^k_j(v),\quad  \mbox{for all $v=(v^J)_{|J|\le r}\in \R^{D}$}.
\]
Here we define $\Psi^k_{\ell_k+1}=0$. 

Now we are ready to define our dilation map $T^{k}$ on $\fI$. For any given block structure $U^k$, for any $u=(u_1,\dots, u_{\ell_k+1})\in \fI$, let
\[
T^{k}_\eta(u)=\sum_{j=1}^{\ell_k+1}\eta^{\gamma^k_j} \Pi^k_j(u)=\sum_{j=1}^{\ell_k+1}\eta^{\gamma^k_j}u_j,
\]
where $\gamma^k_{\ell_k+1}=0$. 
It induces a dilation on the path space $C([0,1],\fI)$ by $(T^{k}_\eta(v))_t:=T^{k}_\eta(v_t)$ for any $v_t$, $0\le t\le 1$. 

Let $y(t)=\sum_{|J|\le r}c^J(W_t,t)X^J\in C([0,1],\fL)$, we know from the stochastic Taylor formula that $x^\ep(t)=\exp_{x_0}(y^\ep(t))$. Let $\hat{y}(t)=\sum_{|J|\le r, J\not=0}c^J(W_t,t)X^J$, then clearly $\hat{y}(t)\in C([0,1],\fI)$. 
We shall prove that $y^\ep$ is $\alpha_k$-exponentially well approximated by $\hat{y}^{\ep}$ for any $1\le k\le\ell$.

\begin{lemma}\label{lemma-y-fl}
$y^\ep$ is $\alpha_k$-exponentially well approximated by $\hat{y}^{\ep}$, i.e. for any $\delta>0$, 
\[
\lim_{\ep\to0}\ep^{2\alpha_k}\log\p(\|y^{\ep}-\hat{y}^{\ep}\|_{[0,1],\infty}>\delta)=-\infty.
\]
\end{lemma}
\begin{proof}
Since when $\ep$ is small enough, we have $\p(\|y^{\ep}-\hat{y}^{\ep}\|_{[0,1],\infty}>\delta)=\p(\|\ep^2t\|_{[0,1],\infty}>\delta)=0<e^{-\frac{1}{2\alpha}}$ for any $\alpha>\alpha_k$. Therefore conclusion holds.
\end{proof}

Our next theorem gives a a graded LDP for the dilated process $y^{k,\ep}(t):=T^{k}_\ep(\hat{y}^\ep(t))$ in $C([0,1],\fI)$. 
\begin{theorem}\label{thm-ldp-y-kep}
The distribution of $y^{k,\ep}$ satisfies a Large Deviation Principle at rate $\ep^{-2\alpha_k}$ with rate function
 \begin{equation}\label{eq-rate-general}
  I^k(\varphi)=-\inf\left(\frac12\|h\|^2_{\h},  \Phi^{k}(c(h_t,t))=\varphi\right).
 \end{equation}
\end{theorem}
\begin{proof}
From our definitions, we have that the $j$-th component of $y^{k,\ep}$ is given by
\[
y^{k,\ep}_j(t)=\ep^{\|J\|+\gamma^k_j}c^J(W_t,t)\,  \Pi^{k}_j(X^J).
\]
By the definition of $U_j^k$ we know that $\Pi^{k}_j(X^J)=0$ if $\gamma^k(J)=n(J)\alpha_k-\|J\|<\gamma^k_j$. Let $\mathcal{C}_j$ be the collection of words $J$ such that $\gamma^k(J)=n(J)\alpha_k-\|J\|>\gamma^k_j$. Then we have
\[
y^{k,\ep}_j(t)=\sum_{J\in \mathcal{B}_j}\ep^{\|J\|+\gamma^k_j}c^J(W_t,t)\,  \Pi^{k}_j(X^J)+\sum_{J\in\mathcal{C}_j}\ep^{\|J\|+\gamma^k_j}c^J(W_t,t)\,  \Pi^{k}_j(X^J).
\]
Using a similar argument as in Lemma \ref{lemma-z}, we can show that $y^{k,\ep}_j(t)$ is $\alpha_k$-exponentially well approximated by $\sum_{J\in \mathcal{B}_j}\ep^{\|J\|+\gamma^k_j}c^J(W_t,t)\,  \Pi^{k}_j(X^J)$, which is in fact $\Phi^k_j  \left(\sum_{X^J\in\fI}c^J(\ep^{\alpha_k}W_t,t)X^J\right)$, since 
\begin{align*}
&\sum_{J\in\mathcal{B}_j}\ep^{\|J\|+\gamma^k_j}c^J(W_t,t)\,  \Pi^{k}_j(X^J)=\sum_{J\in \mathcal{B}_j}c^J(\ep^{\alpha_k}W_t,t)\,  \Pi^{k}_j(X^J)\\
&\quad\quad=\Pi^k_j \circ \Psi^k_j \bigg(\sum_{|J|\le r}c^J(\ep^{\alpha_k}W_t,t) X^J \bigg)=\Phi^k_j  \bigg(\sum_{|J|\le r}c^J(\ep^{\alpha_k}W_t,t) X^J \bigg),
\end{align*}
where $c^J(\ep^{\alpha_k}W_t,t):=\sum_{\sigma\in\sigma_{|J|}}\frac{(-1)^{e(\sigma)}}{|J|^2
{|J|-1\choose e(\sigma)}}I_{J\circ \sigma^{-1}}(\ep^{\alpha_k}W_t,t) $. By the contraction principle, we know that 
$\Phi^k \bigg(\sum_{|J|\le r}c^J(\ep^{\alpha_k}W_t,t)X^J \bigg)$ satisfies a classical LDP at rate $\ep^{-2\alpha_k}$ with rate function \eqref{eq-rate-general}. Which completes the proof.
\end{proof}

At last we want to obtain graded large deviation estimates for $x(t)$ (started from $x_0$). As before, let us first define the notions of closed graded dilation and open graded dilation of a set of paths.
\begin{definition}
For any Borel set $B\subset C([0,1],\fI)$, let
\begin{equation}\label{eq-gr-cl-inte}
\cl^{k}(B)=\cap_{\delta>0}\overline{\cup_{\eta\le\delta}\,T_\eta^{k}(B)},
\quad
\inte^{k}(B)=\cup_{\delta>0}\accentset{\circ}{\arc{\cap_{\eta<\delta}\,T^{k}_{\eta}(B)}}.
\end{equation}
be the closed graded dilation of $B$ and the open graded dilation of $B$. 
\end{definition}

We now state our graded large deviation estimates for the distribution of a general nilpotent diffusion $x(t)$ on $M$. 
\begin{theorem}\label{mainresult3}
There exists an integer $\ell \ge 1$, rational numbers $1=\alpha_1<\cdots<\alpha_\ell$, such that,
for any Borel set $A \subset C([0,1],M)$, we have that 

\begin{equation}\label{eq-main-nil-gr-1}
\liminf_{\ep\to0}\ep^{2\alpha_k}\log\p(x^\ep\in A)\ge -\inf\left( I^k(\varphi), \varphi\in \inte^k(\fI(\exp_{x_0}^{-1}(A)))\right)
\end{equation}
and
\begin{equation}\label{eq-main-nil-gr-2}
\limsup_{\ep\to0}\ep^{2\alpha_k}\log\p(x^\ep\in A)\le -\inf\left( I^k(\varphi), \varphi\in \cl^k(\fI(\exp_{x_0}^{-1}(A)))\right),
\end{equation}
where $\fI(\exp^{-1}_{x_0}(A))=\exp^{-1}_{x_0}(A)\cap C([0,1],\fI)$. 
\end{theorem}

\begin{proof}
By stochastic Taylor formula we know that $\p(x^\ep\in A)=\p(y^\ep\in\exp^{-1}_{x_0}(A))$. By Lemma \ref{lemma-y-fl} we know that $y^\ep$ is $\alpha_k$-exponentially well approximated by $\hat{y}^{\ep}$. Hence we just need to estimate $\p(\hat{y}^{\ep}\in\exp^{-1}_{x_0}(A))$. Note the fact that $\hat{y}^{\ep}\in C([0,1],\fI)$, we have
\[
\p(\hat{y}^{\ep}\in\exp^{-1}_{x_0}(A))=\p(\hat{y}^{\ep}\in\fI(\exp^{-1}_{x_0}(A)))=\p\bigg(y^{\ep,k}\in T^{k}_{\ep}(\fI(\exp^{-1}_{x_0}(A)))\bigg).
\]
Using Theorem \ref{thm-ldp-y-kep} we then have
the conclusion \eqref{eq-main-nil-gr-1} and \eqref{eq-main-nil-gr-2} following exactly the same arguments as in the proof of Theorem \ref{thm-ldp-Y-main}. To avoid repetition, we omit the details here.
\end{proof}

\begin{Remark}\label{rmk-U-invariant}
We want to emphasis here that $T^k$, $\Phi^k$, $\cl^k(\cdot)$ and $\inte^k(\cdot)$ all depend on the choice of the block structure $U^k$. However, the large deviation estimates in Theorem \ref{mainresult3} are independent of the choice of $U^k$. More precisely, we can  introduce the maps $\Theta^{+}_k, \Theta^{-}_k:\bigg( C([0,1],M), \|\cdot\|_{[0,1],\infty}\bigg) \to \bigg(C([0,1],\R^D),, \|\cdot\|_{[0,1],\infty}\bigg)$, such that for any $A\subset C([0,1],M)$, 
\begin{align}\label{eq-theta}
\Theta^{+}_k(A)=\left(\Phi^{k}\right)^{-1}\left(  \cl^k(\fI(\exp_{x_0}^{-1}(A)))\right),\quad
\Theta^{-}_k(A)=\left(\Phi^{k}\right)^{-1}\left(  \inte^k(\fI(\exp_{x_0}^{-1}(A)))\right).
\end{align}
We can easily show that  $\Theta^{+}, \Theta^{-}$ do not depend on the choice of the block structure. 
To see this, consider two different block structures $U^k$ and $\hat{U}^k$ that are both adapted to the flag $V^k$ of $\fL$. Let $\cl^k$, $\inte^k$ and $\hat{\cl}^k$, $\hat{\inte}^k$ denote the corresponding graded dilations. From the construction of $U^k$ and $\hat{U}^k$ we know there exists an invertible map $\mathcal{S}$ such that 
\[
{\Pi}^{\hat{U}^k}_j=\mathcal{S}\circ \Pi^{{U}^k}_j,\quad \forall j=1,\dots, \ell_k+1.
\]
Hence $T^{\hat{U}^k}_\eta=\sum_{j=1}^{\ell_k+1}\eta^{\gamma^k_j}\Pi^{\hat{U}^k}_j=\mathcal{S}\circ T^{{U}^k}_\eta$, which implies that for $B=\fI(\exp_{x_0}^{-1}(A))$
\begin{equation}\label{eq-inv-1}
\hat{\inte}^k(B)=\mathcal{S}\circ \inte^k(B),\quad \hat{\cl}^k(B)=\mathcal{S}\circ \cl^k(B).
\end{equation}
On the other hand, we have
\begin{equation}\label{eq-inv-2}
\Phi^{\hat{U}^k}=\sum_{j=1}^{\ell_k+1} \mathcal{S}\circ \Pi^{U^k}_j\circ \Psi_j=\mathcal{S}\circ \Phi^{{U}^k}.
\end{equation}
By combining \eqref{eq-inv-1}, \eqref{eq-inv-2} and \eqref{eq-theta} we obtain the conclusion. 
\end{Remark}

We then have an alternative expression of the large deviation estimates \eqref{eq-main-nil-gr-1} and \eqref{eq-main-nil-gr-2} in Theorem \ref{mainresult3}. 
\begin{equation}\label{eq-main-nil-al-1}
\liminf_{\ep\to0}\ep^{2\alpha_k}\log\p(x^\ep\in A)\ge -\inf\left(\frac{1}{2}\|h\|^2_{\h},c(h_t,t)\in \Theta_k^-(A)\right)
\end{equation}
and
\begin{equation}\label{eq-main-nil-al-2}
\limsup_{\ep\to0}\ep^{2\alpha_k}\log\p(x^\ep\in A)\le -\inf\left(\frac{1}{2}\|h\|^2_{\h},c(h_t,t)\in \Theta_k^+(A)\right).
\end{equation}


\begin{refproof}[Proof of Theorem \ref{mainresult1} and \ref{mainresult2}]
The grades $1=\alpha_1<\cdots<\alpha_\ell$ can be found as in \eqref{eq-G-seq}. We just need to construct $B_k\subset M$ such that the corresponding $A_k = \{ \phi \in C([0,1],M), \phi(1) \in B_k \}$ of grade $\alpha_k$ with respect to $\p^\ep$. Consider $B_k=\exp_{x_0}(C_k)$ where 
\[
C_k=\bigg(\phi\in C([0,1], W(\alpha_k)), |\Pi^{W(\alpha_{k-1})}\phi(1)-\phi(1)|>1  \bigg).
\]
\end{refproof}

\section{Examples}\label{sec-example}
\subsection{Theorem \ref{mainresult3} at grade $\alpha_1$}
In this section we discuss the comparison of Theorem \ref{mainresult3}  with the classical Large Deviation Principle (Theorem \ref{Azencott}). 
\begin{proposition}\label{prop-compare}
The large deviation estimates in Theorem \ref{mainresult3} with grade $\alpha_1=1$ implies Azencott's Large Deviation Principle in Theorem \ref{Azencott}. 
\end{proposition}
\begin{proof}
We start from Theorem \ref{mainresult3} at grade $\alpha_1=1$.
Clearly for all  $|J|\le r$, $\gamma^1(J)\le0$. Hence $\gamma^1_1=0$, and we have $V^1_1=\Span\{X^J, \alpha(J)=1\}= W(\alpha_1)$.  The flag of $\fL$ is simply
\[
V^1_1\subset \fL.
\]
We have $\mathcal{B}_1=\{J, \alpha(J)=1\}$ and $\Psi_1(c)=\sum_{\alpha(J)=1}c^JX^J$. The block structure $U^1$ is 
\[
U^1_1\oplus U^1_2=\fL
\]
where $U^1_1=V^1_1$. In particular if $X_0\in V^1_1$ then $U^1_2=\emptyset$. We then have the dilation $T^{\alpha_1}_\eta=\rm{Id}$ and hence the graded dilations
\[
\cl^{1}(A)=\overline{\exp^{-1}_{x_0}(A)},\quad \inte^1(A)=\accentset{\circ}{\arc{\exp^{-1}_{x_0}(A)}}
\]
for all $A\subset C([0,1],\fL)$.
Note $X^J\in U^1_1$ for all $J$ such that $\alpha(J)=1$, we have $\Phi^{1}(v)=\sum_{\alpha(J)=1}v^JX^J$. From \eqref{ode} and \eqref{eq-taylor-F} we can easily check that $\exp_{x_0}(\Phi^{1}(c(h_t,t)))=\Phi_{x_0}(h)(t)$.
Then Theorem \ref{mainresult3} can be stated as follows. For any $A\subset C([0,1],M)$, 
\begin{equation}\label{eq-dis-1}
\liminf_{\ep\to0}\ep^{2}\log\p(x^\ep\in A)\ge -\inf\left(\frac{1}{2}\|h\|^2_{\h},\Phi^{1}(c(h_t,t))\in  \accentset{\circ}{\arc{\exp_{x_0}^{-1}(A)}}\right)
\end{equation}
and
\begin{equation}\label{eq-dis-2}
\limsup_{\ep\to0}\ep^{2}\log\p(x^\ep\in A)\le -\inf\left(\frac{1}{2}\|h\|^2_{\h},\Phi^{1}(c(h_t,t)) \in 
\overline{\exp_{x_0}^{-1}(A)}\right).
\end{equation}
We therefore complete the proof by noting that $\exp^{-1}(\ring{A})\subset\accentset{\circ}{\arc{\exp^{-1}(A)}}$ and $\exp^{-1}(\overline{A})\supset\overline{\exp^{-1}(A)}$.
\end{proof}

\subsection{Graded Large Deviations for the Kolmogorov process}\label{sec-kol}
\subsubsection{Theorem \ref{mainresult3} for the Kolmogorov process}
In this section we apply Theorem \ref{mainresult3} to the Kolmogorov process $x(t)$, as defined in Example \ref{ex-kol}. Assume it  starts from point $(x^1_0,x^2_0)$, then
\begin{equation}\label{eq-kol}
x(t)=\left(x^1_0+w_t,x^2_0+x^1_0t+\int_0^tw_sds \right).
\end{equation}
\begin{theorem}\label{thm-kol}
The distribution of the Kolmogorov process $x(t)$ satisfies  graded large deviation estimates at grades $\alpha_1=1$ and $\alpha_2=3$. For any $A\subset C([0,1], \R^2)$ equipped with $\|\cdot\|_{[0,1],\infty}$ norm, 
\begin{itemize}
\item[(1)] At grade $\alpha_1=1$, we have
\begin{equation}\label{eq-kol-grade-1-sup}
\limsup_{\ep\to0}\ep^{2}\log\p(x^\ep\in {A})\le -\inf\left(\frac{1}{2}\|h\|^2_{\h},(h_t,0) \in \exp^{-1}_{x_0}(\overline{A}) \right)
\end{equation}
and
\begin{equation}\label{eq-kol-grade-1-inf}
\liminf_{\ep\to0}\ep^{2}\log\p(x^\ep\in {A})\ge -\inf\left(\frac{1}{2}\|h\|^2_{\h},(h_t,0) \in \exp^{-1}_{x_0}(\ring{A})\right).
\end{equation}
\item[(2)] At grade $\alpha_2=3$, we have
\begin{equation}\label{eq-kol-grade-2-sup}
\limsup_{\ep\to0}\ep^{6}\log\p(x^\ep\in A)\le -\inf\left(\frac{1}{2}\|h\|^2_{\h}, \left(h_t,  \int_0^th_sds-\frac12th_t,0\right)\in \cl^{2}\fI((\exp^{-1}_{x_0}(A)))\right)
\end{equation}
and
\begin{equation}\label{eq-kol-grade-2-inf}
\liminf_{\ep\to0}\ep^{6}\log\p(x^\ep\in A)\ge -\inf\left(\frac{1}{2}\|h\|^2_{\h}, \left(h_t,  \int_0^th_sds-\frac12th_t,0\right)\in \inte^{2}\fI((\exp^{-1}_{x_0}(A)))\right).
\end{equation}
\end{itemize}
\end{theorem}
\begin{proof}
Clearly $\fL$ has grading $\alpha_1=1$, $\alpha_2=3$, and 
\[
W(\alpha_1)\subsetneq W(\alpha_2)=\fI\subsetneq \fL. 
\] 
where $W(\alpha_1)=\Span\{ X_1\}$ and $W(\alpha_2)=\Span\{ X_1, [X_1,X_0]\}$. Recall from \eqref{eq-taylor-F} that $c^{(1)}(h_t,t)=h_t$, $c^{(1,0)}(h_t,t)=\frac12\int_0^th_sds$, $c^{(0,1)}(h_t,t)=\frac12\int_0^tsdh_s$ and $c^{(0)}(h_t,t)=t$. 

(1) For grade $\alpha_1=1$, we have 
$V^1_1=W(\alpha_1)\subsetneq \fI, \quad \gamma^1_1=0$.
The block structure is simply 
\[
\fI=U^1_1\oplus U^1_2
\] where $U^1_1=\Span\{ X_1\}$ and $U^1_2=\Span\{[X_1, X_0]\}$. The corresponding map $\Phi^1$ for any $c=(c^{(1)}, c^{(1,0)}, c^{(0,1)}, c^{(0)})\in C([0,1],\R^4)$ is given by
$\Phi^1(c)=(c^{(1)},0,0)$. 
Hence we have 
\[
I^{1}(\varphi)=-\inf\left(\frac12\|h\|^2_{\h},h_t=\varphi\right).
\]
Moreover, since the corresponding $T^{1}_\eta$ dilation is the identity map, we have $\cl^{1}(B)=\overline{B}$ and  $\inte^{1}(B)=\ring{B}$ for all $B\subset C([0,1], \fI)$. Hence Theorem \ref{mainresult3} at grade $\alpha_1=1$ implies that for any Borel set $A\subset C([0,1], \R^2)$,
\[
\limsup_{\ep\to0}\ep^{2}\log\p(x^\ep\in {A})\le -\inf\left(\frac{1}{2}\|h\|^2_{\h},(h_t,0,0) \in \overline{\fI(\exp^{-1}_{x_0}(A))} \right)
\]
and
\[
\liminf_{\ep\to0}\ep^{2}\log\p(x^\ep\in {A})\ge -\inf\left(\frac{1}{2}\|h\|^2_{\h},(h_t,0,0) \in \accentset{\circ}{\arc{{\fI}(\exp^{-1}_{x_0}(A)})}\right).
\]
At last note that $\overline{\fI(\exp^{-1}_{x_0}(A))}\subset \fI(\exp^{-1}_{x_0}(\overline{A}))$ and $\accentset{\circ}{\arc{{\fI}(\exp^{-1}_{x_0}(A)})}\supset \fI(\exp^{-1}_{x_0}(\ring{A}))$, and the fact that 
\[
\bigg( (h_t,0,0) \in \fI (B)\bigg)=\bigg( (h_t,0,0) \in B\bigg)
\]
for any $B\subset C([0,1],\fL)$, we obtain the conclusion in \eqref{eq-kol-grade-1-sup} and \eqref{eq-kol-grade-1-inf}.

(2) For grade $\alpha_2=3$, we have the secondary flag structure giving by
\[
V^2_1\subsetneq V^2_2=\fI\subsetneq\fL,\quad 2=\gamma^2_1>\gamma^2_2=0
\]
where $V^2_1=\Span\{X_1 \}$ and $V^2_2=\Span\{X_1, [X_1, X_0] \}$. We have $\mathcal{B}_1=\{(1)\}$, $\mathcal{B}_2=\{(1,0), (0,1)\}$. Also
\[
\Psi^2_1(c)=c^{(1)}X_1,\quad \Psi^2_2(c)=c^{(1,0)}[X_1,X_0]+c^{(0,1)}[X_0,X_1]. 
\]
We take the block structure $\fI=U^2_1\oplus U^2_2$ where
\[
U^2_1=\Span\{ X_1\}, \quad U^2_2=\Span\{ [X_1, X_0]\}.
\]
Let $U^2_0=\Span\{ X_0\}$, then $\fL=\fI\oplus U^2_0$. The dilation $T^{2}_\eta$ on $\fI$ is given by $T^2_\eta(v)=\eta^2\Pi_1^{2}(v)+\Pi_2^{2}(v)$ for any $v\in \fI$.  
We also have \[
\Phi^{2}(c)=\bigg(c^{(1)}, c^{(1,0)}-c^{(0,1)},0\bigg).
\]
Theorem \ref{mainresult3} then implies \eqref{eq-kol-grade-2-sup} and \eqref{eq-kol-grade-2-inf}.
\end{proof}

We apply the above theorem to the Example \ref{ex-kol} where $x_0=0$ and $B_2= \{ (x^1,x^2) \in \R^2, x^2 > 1 \}$. We show now that our estimate implies the estimates stated in \eqref{eq-kol-Q-B-2}.
\addtocounter{example}{-2}
\begin{example}[Proof of the estimate \eqref{eq-kol-Q-B-2}] 
 By Theorem \ref{thm-kol} we have at grade $\alpha_2=3$, 
\begin{equation}\label{eq-kol-ex}
\limsup_{\ep\to0}\ep^{6}\log Q^\ep(B_2)\le -\inf\left(\frac{1}{2}\|h\|^2_{\h}, \left(h_1,  \int_0^1h_sds-\frac12h_1,0\right)\in \cl^{2}\fI((\exp^{-1}_{0}(B_2)))\right).
\end{equation}
Note $\left(h_t, \int_0^th_sds-\frac12th_t,0 \right)\in \cl^{2}(\fI(\exp^{-1}_{0}({B_2})))$ if and only if for any $\delta>0$, there exists a sequence $\eta_n<\delta$ and $(f_n,\ell_n, 0)\in \fI(\exp^{-1}_{0}(B_2))$, such that
\[
\lim_{n\to\infty}\eta_n^2 f_n=h_t, \quad  \lim_{n\to\infty}\ell_n=\int_0^th_sds-\frac12th_t.
\]
Note $(f_n,\ell_n,0)\in \fI(\exp^{-1}_{0}(B_2))$ means that $(f_n,\ell_n)|_{t=1}\in B_2$, i.e. $\ell_n(1)>$ for all $n\ge 1$. This implies that $ \cl^{2}(\fI(\exp^{-1}_{0}({B_2})))\subset \bigg(h\in \h, \int_0^1h_sds\ge1+\frac12h_1 \bigg)$. Hence \eqref{eq-kol-ex} implies that
\[
\limsup_{\ep\to0}\ep^{6}\log Q^\ep(B_2)\le -\inf\left(\frac{1}{2}\|h\|^2_{\h},\int_0^1h_sds\ge1+\frac12h_1  \right)=-\frac32.
\]

(2) By \eqref{eq-kol-grade-2-inf} we have \begin{equation}\label{eq-kol-ex-b-2-inf}
\liminf_{\ep\to0}\ep^{6}\log\p(x^\ep\in B_2)\ge -\inf\left(\frac{1}{2}\|h\|^2_{\h}, \left(h_1,  \int_0^1h_sds-\frac12h_1,0\right)\in \inte^{2}\fI((\exp^{-1}_{0}(B_2)))\right).
\end{equation}
Note $\left(h_t, \int_0^th_sds-\frac12th_t,0 \right)\in \inte^{2}(\fI(\exp^{-1}_{0}({B_2})))$ if and only if there exists a $\delta>0$, and $\rho>0$ such that for all $(f,g, 0)\in C([0,1],\fI)$ satisfying 
 \[
 \|f-h\|_{[0,1],\infty}<\rho,\quad  \bigg\|g-\left(\int_0^t h_sds-\frac12 th_t\right)\bigg\|_{[0,1],\infty}<\rho,
 \]
we have $(f,g,0)\in T^{2}_{\eta}(\fI(\exp^{-1}_0(B_2)))$ for all $\eta\le\delta$, i.e.
 $ g_1 >1$ for all $\eta\le\delta$. Hence we have
 \[
 \inte^{2}(\fI(\exp^{-1}_{0}({B_2})))=\left(h\in \h, \int_0^1 h_sds-\frac12 h_1>1\right). 
 \]
 Therefore 
 \[
\liminf_{\ep\to0}\ep^{6}\log Q^\ep(B_2)\ge -\inf\left(\frac{1}{2}\|h\|^2_{\h},\int_0^1h_sds>1+\frac12h_1  \right)=-\frac32.
\]
All together we obtain \eqref{eq-kol-Q-B-2}.
\end{example}

\subsection{A potential reformulation for large deviation estimates for very rare events}
As mentioned earlier, it is a natural question to ask whether one can develop a large deviation estimate for a general nilpotent diffusion $x(t)$ of the following form. For $A\subset C([0,1],M)$, 
\[
\limsup_{\ep\to0}\frac{\log\p(x^\ep\in \overline{A})}{I^\ep(\overline{A})}\le -1,\quad \liminf_{\ep\to0}\frac{\log\p(x^\ep\in \ring{A})}{I^\ep(\ring{A})}\ge -1
\]
where the rate function $I^\ep$ is $\ep$-dependent,
\begin{equation}\label{eq-I-Psi-ep}
I^\ep(\cdot)=\inf\left(\frac12\|h\|_{\h}^2, \Psi_{x_0}^\ep(h)\in \cdot \right), \quad d\Psi_{x_0}^\ep(h)(t)=\ep\sum_{i=1}^mX_i(\Psi_{x_0}^\ep(h))dh_t^i+\ep^2X_0(\Psi_{x_0}^\ep(h))dt.
\end{equation}
The answer is yes for some sets $A$. But there are also sets for which the rate function $I^\ep$ is never finite.
 In this section, we use several examples on the Kolmogorov process to illustrate these aspects.
\subsubsection{Variational computation for Kolmogorov process}\label{sec-variation-kol}

Consider the Kolmogorov process $x(t)=\left(w_t,\int_0^tw_sds\right)$. Since it is a Gaussian process, we can easily obtain its density (see \cite{CME}),
 \[
 p_{\ep^2}((0,0),(x^1,x^2))
 =
 \frac{\sqrt{12}}{2\pi\ep^4}\exp\bigg\{-\frac{1}{2}\left[\frac{4}{\ep^2}(x^1)^2-\frac{12}{\ep^4}x^1x^2+\frac{12}{\ep^6}(x^2)^2 \right] \bigg\}.
 \]
Namely we have
 \begin{equation}\label{eq-D}
\lim_{\ep\to0}\frac{ \log p_{\ep^2}((0,0),(x^1,x^2))}{D^\ep}=-1
, \quad D^\ep=\frac{2}{\ep^2}(x^1)^2-\frac{6}{\ep^4}x^1x^2+\frac{6}{\ep^6}(x^2)^2.
 \end{equation}
In fact $D^\ep$ is the solution of the sub-Riemannian control problem. Let $\Psi^\ep_{0}$ be as given in \eqref{eq-I-Psi-ep} with $x_0=0$, our proposition below shows that $D^\ep$ is indeed the ``minimal energy" for $\Psi^\ep_{0}$ to be at point $(x^1, x^2)$ at time $1$. 
 \begin{proposition}
For any $(x^1, x^2)\in \R^2$, we have
\[
\inf\left(\frac{1}{2}\|h\|_{\h}^2,\Psi^\ep_{0}(h)(1)=(x^1,x^2)\right)=D^\ep.
\]
The minimum is achieved at $h_t=\frac{6x^2}{\ep^3}(t-t^2)+\frac{x^1}{\ep}(3t^2-2t)$, $t\in[0,1]$ and the optimal path is given by
\begin{equation}\label{eq-kol-optimal-path}
\Psi^\ep_{0}(h)(t)=\bigg(\frac{6x^2}{\ep^2}(t-t^2)+x^1(3t^2-2t), x^2(3t^2-2t^3)+\ep^2 x^1(t^3-t^2)   \bigg).
\end{equation}
\end{proposition} 

\begin{proof}
This is a simple optimization problem: minimize $\|h\|_{\h}^2$ under the constraint
\[
\Psi^\ep_{0}(h)(1)=\left(\ep h_1,\ep^3\int_0^1h_sds \right)=(x^1,x^2).
\]
We have
\begin{equation}\label{eq-initial}
h_1=\frac{x^1}{\ep},\quad \int_0^1h_tdt=\frac{x^2}{\ep^3}.
\end{equation}
Note that any critical point of $\|h\|_{\h}^2$ under the linear constraints has to be quadratic in time. Therefore we can assume 
\[
h_t=at^2+bt,
\]
such that
\[
h_1=a+b, \quad \int_0^1h_sds=\frac{a}{3}+\frac{b}{2}.
\]
By plugging into \eqref{eq-initial} we obtain that
\[
\begin{cases}
a=-\frac{6x^2}{\ep^3}+\frac{3x^1}{\ep}
\\
b=\frac{6x^2}{\ep^3}-\frac{2x^1}{\ep}.
\end{cases}
\]
Hence we have
$h_t=\frac{6}{\ep^3}x^2(t-t^2)+\frac{x^1}{\ep}(3t^2-2t)$ and \eqref{eq-kol-optimal-path}. We can then compute
\[
\frac{1}{2}\|h\|_{\h}^2=\frac{6(x^2)^2}{\ep^6}-\frac{6x^1x^2}{\ep^4}+\frac{2(x^1)^2}{\ep^2}
\]
which agrees with  \eqref{eq-D}. The proof is then complete.
\end{proof}
\begin{Remark}It is important to emphasize the potential non-local character of this variational problem. It can indeed happen that finding the optimal path is not a local problem. 
\begin{itemize}
\item[(1)] When $x^2=0$, the optimal path converges to $\Phi_0(h)(t)=(x^1(3t^2-2t), 0)$. When the target point $(x^1,0)$ is close to the starting point $(0,0)$, the optimal path stays in a neighborhood of these points. The reason is indeed that the point $(x^1,0)$ is horizontally accessible. The process $x^\ep$ only needs to move along the admissible direction $\ep X^1=\ep\frac{\partial}{\partial x^1}$ to attain the target point.  Also there exists a horizontal path of minimal (finite) energy connect to $(x^1,0)$. A classical LDP then tells us that $\p^\ep(x(1)=(x^1,0))$ concentrates around this optimal path as $\ep\to0$.
\item[(2)] When $x^2\not=0$, the optimal path \eqref{eq-kol-optimal-path} diverges away as $\ep$ tends to $0$, and is not confined. Indeed the reason is that the target is no longer horizontally accessible.The process $x^\ep$ needs the help from the drift $\ep^2X_0=\ep^2x^1\frac{\partial}{\partial x^2}$ to make its vertical displacement. However, the magnitude of the drift is extremely small as $\ep\to0$, unless the diffusion can make a horizontal displacement of size $\frac{1}{\ep^2}$ to offset this small magnitude. This explains why the optimal path horizontally diverges to infinity  as $\ep\to0$. 
\end{itemize}\end{Remark}

\subsubsection{Example of a ``bad" set}\label{sec-bad-set}
In this section we given an example to illustrate that it is not possible to develop a large deviation estimate of the form \eqref{eq-I-Psi-ep}  for processes satisfying a weak H\"ormander's condition, even for very simple ones like the Kolmogorov process. 
\begin{proposition}
Let $x(t)=\left(w_t,\int_0^tw_sds \right)$. There exists a $C\subset C([0,1], \R^2)$ such that 
 \[
 \lim_{\ep\to0}\ep^6\log\p(x^\ep \in C)=-2\quad  \mbox{but }\quad I^\ep(C)=-\infty.
 \] 
\end{proposition}
\begin{proof}
Let $C=A\times B$ be a product set in $C([0,1],\R)^2$ where $B=\{g\in C([0, 1], \R), g(1) \ge 1\}$ and $A$ is given as below.

Let $\{f_i\}_{i\ge0}$ be an orthonormal basis in $L^2([0,1],\R)$ of smooth functions. In particular we let $f_0(s)=1-s$. For any continuous function $w(s)$ such that $w(0)=0$, we consider
\[
Z_i(w)=\int_0^1f_i(s)dw(s).
\]
Then $Z_i$ are i.id $N(0,1)$ under Wiener measure. Moreover, for any $p\in\Z_+$,  $\sum_{i=1}^pZ_i^2$ is a $\chi^2{(p-1)}$ random variable in distribution whose density is given by $\frac{1}{\Gamma(p/2)}e^{-u/2}u^{p/2-1}$. We take $p=4$ then 
\[
\p(\chi^2(3)\ge v)=\frac{1}{\Gamma(3)}\int_v^\infty e^{-\frac{u}{2}}udu.
\] 
Let 
\begin{equation}\label{eq-bad-set}
A_i=\bigg\{w(\cdot):Z_{4i-3}^2+Z_{4i-2}^2+Z_{4i-1}^2+Z_{4i}^2\ge\frac{1}{i}\bigg\}, \quad A=\cap_{i\ge1}A_i.
\end{equation}
Observe that $A_i$ are independent, therefore 
\[
\p(A)=\Pi_{i\ge 1}\p(A_i)=\Pi_{i\ge 1}\frac{1}{4}\int_{1/i}^\infty e^{-\frac{u}{2}}udu.
\]
Let $\p^\ep(A)=\Pi_{i\ge1}\p^\ep(A_i):=\Pi_{i\ge1}\p(A_i^\ep)$
where $A_i^\ep=\{\ep w(\cdot):(Z_{4i-3}^2+Z_{4i-2}^2+Z_{4i-1}^2+Z_{4i}^2)\ge\frac{1}{i}\}$,
then
\[
\p^\ep(A)=\Pi_{i\ge 1}\frac{1}{4}\int_{1/i\ep^2}^\infty e^{-\frac{u}{2}}udu.
\]
Obviously $\p^\ep(A)>0$ and 
$
\limsup_{\ep\to0}\ep^2\log\p^\ep(A)=-\infty.
$
In fact $-\log\p^\ep(A)$ grows faster than $\ep^{-2}$ with an extra $\log$ factor.

However, notice that set $A$ is closed in $C([0,1],\R)$ under uniform topology and contains no Cameron-Martin path. This is because for any $h\in\h$ we have $\sum_{i\ge 1}Z_i^2(h)<+\infty$.  Hence
\[
I^\ep(C)=\inf\left(\frac12\|h\|_{\h}^2, \ep h\in A, \ep^3 \int_0^1 h_tdt\ge1 \right)=-\infty.
\]
On the other hand, since
\begin{align*}
\p\left(x^\ep\in C\right)=\p\left(\ep w \in A,\ep^3\int w_sds\in B\right)
=\p^\ep( A)\,\p\left( \ep^3\int w_sds\in B\right),
\end{align*}
we know for some constant $K\in\R$,
\[
-{\frac{\ep^4}{\log(1/\ep)}K}+\ep^6\log\p\left(\ep^3 \int_0^t w_sds \in B\right)\le \ep^6\log \p_\ep\left(C \right)\le
{\frac{\ep^4}{\log(1/\ep)}K}+\ep^6\log\p\left(\ep^3 \int_0^t w_sds \in B\right).
\]
Hence
\begin{align}\label{eq-kolm-rate}
\lim_{\ep\to0}\ep^6\log \p\left(x^\ep \in C \right)=
\lim_{\ep\to0}\ep^6\log \p\left(\ep^3\int_0^1  w_sds \ge1 \right) =-\frac32.
\end{align}
\end{proof}
As a comparison, we apply Theorem \ref{thm-kol} at grade $\alpha_2=3$ to the above example. Though there is no horizontal path in $C$, our graded large deviation estimate still provide a reasonable upper bound. 
\begin{proposition} Let $C\subset C([0,1],\R^2)$ be given as above. We have
\[
\limsup_{\ep\to0}\ep^6\log\p(x^\ep\in C)\le-\inf\left(\frac{1}{2}\|h\|_{\h}^2,\int h \in B \right)=-\frac32.
\]
\end{proposition}
\begin{proof}
From \eqref{eq-kol-grade-2-sup} 
we have
\[
\limsup_{\ep\to0}\ep^6\log\p(x^\ep\in C)\le -\inf\left(\frac{1}{2}\|h\|^2_{\h},\left(h_t,\int_0^t h_sds-\frac12th_t,0\right)\in \cl^{2}(\fI(\exp_{0}^{-1}(C)))\right).
\]
We claim that
\begin{equation}\label{eq-gLDP-rate}
\inf\left(\frac{1}{2}\|h\|^2_{\h},\left(h_t,\int_0^t h_sds-\frac12th_t,0\right)\in \cl^{2}(\fI(\exp_{0}^{-1}(C)))\right)=\inf\left(\frac{1}{2}\|h\|_{\h}^2,\int h \in B \right).
\end{equation}

First, let us consider the minimizer of $\inf\left(\frac{1}{2}\|h\|_{\h}^2,\int h \in B \right)$ and denote it by $\fh\in \h$. We want to show that $\left(\fh_t,\int_0^t \fh_sds-\frac12t\fh_t,0\right)\in\cl^{\alpha_2}(\fI(\exp_{0}^{-1}(C)))$. Due to the fact that $B$ is  closed and $\|\cdot\|_{\h}: \h([0,1], \R)\to [0,+\infty)$ is lower semicontinuous, we know that $\int \fh\in B$. We just need to prove that there exists a sequence $(h_n, g_n,0) \in\fI(\exp_{0}^{-1}(C))$ where $h_n\in A, g_n\in B$ and $\ep_n>0$ such that 
 \begin{equation}\label{eq-h_n}
 \lim_{n\to\infty}\bigg\|\left(\ep_n^2h_n,g_n-\frac{1}{2}h_nt,0\right)-\left(\fh_t,\int_0^t \fh_sds-\frac12\fh_tt,0\right)\bigg\|_{[0,1],\infty}=0.
 \end{equation}
We construct $h_n$ as follows. Let $k\in A$ and $\|k\|_{[0,1],\infty}=M<\infty$, then $\|k\|_{\h}=+\infty$.  We assume that $Z_i(\fh)\not=0$ for $1\le i\le m$, let
{
\[
\hat{k}=\int \sum_{i=1}^mZ_i(k)f_i,
\]
}
then $Z_i(k-\hat{k})=0$ for all $i\le m$. We now let
$
h_n=k-\hat{k}+\frac{1}{\ep_n^2} \fh.
$
It suffice to prove 
\begin{itemize}
\item[(a)] $h_n\in A$.
\item[(b)] $\lim_{n\to\infty}\|\ep_n^2 h_n-\fh\|_\infty=0$.
\item[(c)] $g_n:=\frac12 h_nt+\int_0^t \fh_sds-\frac12\fh_tt\ge 1$.
\end{itemize}

To prove (a) we first denote $U_i(\cdot)=Z_{4i-3}^2(\cdot)+Z_{4i-2}^2(\cdot)+Z_{4i-1}^2(\cdot)+Z_{4i}^2(\cdot)$, since for any $i> \left[\frac{m}{4}\right]$,
\[
U_i(k-\hat{k})=U_i(k)\ge\frac{1}{i},
\]
thus $
U_i(h_n)\ge\frac{1}{i}$ for all $i> \left[\frac{m}{4}\right]$.
For $i\le  \left[\frac{m}{4}\right]$, since $U_i(\fh)>0$, we just need choose $\ep_1$ (first term of the decreasing sequence $\{\ep_n\}_{n\ge1}$) small enough such that
$
U_i(\fh/\ep_1^2)\ge\frac{1}{i}.
$
Then we have for all $1\le i\le  \left[\frac{m}{4}\right]$,
\[
U_i(h_n)\ge U_i(\fh/\ep_n^2)\ge\frac{1}{i}.
\]
Hence (a) is proved. Now to prove (b) we just need to observe that
\[
\|\ep_n^2 h_n-\fh\|_\infty=\ep_n^2\|k-\hat{k}\|_\infty.
\]
Since $\|k\|_\infty<\infty$, claim (b) easily follows. To see (c), we just need to realize that $h_n\ge \fh_t$ when $n$ is large enough. 

Therefore we have the minimizer $\left(\fh_t,\int_0^t \fh_sds-\frac12t\fh_t,0\right)\in\cl^{2}(\fI(\exp_{0}^{-1}(C)))$. Hence 
\[
\inf\left(\frac{1}{2}\|h\|^2_{\h},\left(h_t,\int_0^t h_sds-\frac12th_t,0\right)\in \cl^{2}(\fI(\exp_{0}^{-1}(C)))\right)\le\inf\left(\frac{1}{2}\|h\|_{\h}^2,\int h \in B \right).
\]

To prove the other direction, note $\left(h_t,\int_0^t h_sds-\frac12th_t,0\right)\in \cl^{2}(\fI(\exp_{0}^{-1}(C)))$ means there exist $(f_n,g_n)\in C$ and $\ep_n>0$ such that 
 \[
 \lim_{n\to\infty}\bigg\|\left(\ep_n^2f_n,g_n,0\right)-\left(h_t,\int_0^t h_sds-\frac12th_t,0\right)\bigg\|_{[0,1],\infty}=0,
 \]
which implies that $\int_0^1 h_sds\ge \frac12h_1+1$. Hence we can easily obtain that $\bigg\{\left(h_t,\int_0^t h_sds-\frac12th_t,0\right)\in \cl^{2}(\fI(\exp_{0}^{-1}(C)))\bigg\}\subset \{\int h\in B\}$. This implies that 
\[
\inf\left(\frac{1}{2}\|h\|^2_{\h},\left(h_t,\int_0^t h_sds-th_t,0\right)\in \cl^{2}(\fI(\exp_{0}^{-1}(C)))\right)\ge \inf\left(\frac{1}{2}\|h\|_{\h}^2,\int h \in B \right).
\]
Hence we have \eqref{eq-gLDP-rate}.
At the end, we can conclude the upper bound for the exponential estimate:
\[
\limsup_{\ep\to0}\ep^6\log\p(x^\ep\in C)\le
-\inf\left(\frac{1}{2}\|h\|_{\h}^2,\int_0^1 h_tdt \ge 1 \right)=-\frac32.
\]
This agrees with the previous estimate in \eqref{eq-kolm-rate}.
\end{proof}
\begin{Remark}
As for the lower bound, $\left(h_t,\int_0^t h_sds-\frac12th_t,0\right)\in \inte^{2}(\fI(\exp_{0}^{-1}(C)))$ means that there exist $\sigma>0$ and $\rho>0$ such that for all $(f,g)$ satisfying 
 \[
 \|f-h\|_{[0,1],\infty}<\rho,\quad  \bigg\|g-\int h\bigg\|_{[0,1],\infty}<\rho,
 \]
we have $(f,g,0)\in T^{2}_{\eta}(\fI(\exp^{-1}_0(C)))$ for all $\eta\le\sigma$, i.e, 
 \[
 \frac{f}{\eta^2} \in A, \quad g(1) \ge1\qquad \mbox{for all $\eta\le\sigma$}.
 \]
 Since $A$ is a bounded set, $\{\cap_{\eta\le\sigma}\eta^2 A\}=\{0\}$ has no interior. There is no information given for the lower bound.
 \end{Remark}

\subsection{Solvable diffusions}\label{sec-solv}
In this section we briefly discuss the large deviation estimate for the very simple solvable diffusion given in Example \ref{ex-solvable}, which is the natural diffusion on the simple affine group. Clearly $x(t)=(w_t, \int_0^t e^{w_s}ds)$ is the solution of \eqref{eq-sde-solv}. The Lie algebra $\fL$ generated by $X_0=e^{x^1} \frac {\partial}{\partial x^2}$ and $X_1=\frac{\partial}{\partial x^1}$ is not nilpotent but solvable. Indeed
$[X_1,\cdots[X_1, [X_,X_0]]]=X_0$ for any number of brackets with $X_1$, but
\[
\fL_1=[\fL,\fL]=\Span\{X_0\}, \ [\fL_1,\fL_1]=0,
\] 
gives the chain $\fL\supset \fL_1$ of step $2$. 

In this case the $\alpha$-indices no longer provides useful information of the grading structure. We still have grade $\alpha_1=1$ for events that are horizontally accessible, but the second grade $\alpha_2$ is not well defined. In fact the second grade of large deviation estimate is not a polynomial of $\ep$, but includes a log factor.
The correct grading for large deviation estimate is $\ep^2$ and  $\left(\frac{\ep}{\log\frac{1}{\ep}}\right)^2$.
In the following example, we exhibit this phenomenon  by a simple estimate of the non-horizontally accessible event $(x^{\ep}(1) \in B_2)$ where $B_2= \{ (x^1,x^2) \in \R^2, x^2 > 1 \}$.
\begin{proposition}\label{lemma-affine}
For any $a>0$, we have
\[
\lim_{\ep\to0}\frac{\ep^2}{\log^2(1/\ep)}\log \p\left( \ep^2\int_0^1 e^{\ep w_s}ds>a\right)=-2.
\]
\end{proposition}
\begin{proof}
First note that
\begin{align*}
& \p\left( \ep^2\int_0^1 e^{\ep w_s}ds>a\right)\le \p\left(  e^{\ep\|w\|_{[0,1],\infty}}>\frac{a}{\ep^2}\right)
= \p\left(  {\ep\|w\|_{[0,1],\infty}}>\log a+\log{\frac{1}{\ep^2}}\right)\\
& \le \p\left(  {\ep\|w\|_{[0,1],\infty}}>\log{\frac{1}{\ep^2}}\right)= \p\left(  \frac{\ep}{\log\frac{1}{\ep^2}}\|w\|_{[0,1],\infty}>1\right)
\approx e^{-\frac{\log^2{1/\ep^2}}{2\ep^2}}
\end{align*}
where $\approx$ denote exponential approximation at the scale of $e^{-\frac{\log^2(1/\ep^2)}{\ep^2}}$.
On the other hand, for any fixed $0<\alpha<\frac12$, consider the event 
\[
E_\delta:=\bigg\{\|\ep w\|_{\alpha}:=\sup_{t,s\in [0,1]}\frac{|w_t-w_s|}{|t-s|^\alpha}<\delta\bigg\}.
\]
Let $t_0$ be such that $w_{t_0}=\max_{[0,1]} w_t$. Then in $E_\delta$, for any $\eta>0$ and $s\in(0,1)$ such that $|s-t_0|<\eta$, 
\[
0<w_{t_0}-w_s<\delta \eta^\alpha.
\]
Hence
\[
\int_0^1e^{\ep w_s}ds\ge \int_{t_0-\eta}^{t_0+\eta}e^{\ep\left(w_{t_0}-\delta\eta^\alpha\right)}ds\ge 
e^{\ep w_{t_0}}\left(2\eta e^{-\delta\eta^\alpha}\right).
\]
We then have
\begin{align*}
&\p\left( \ep^2\int_0^1 e^{\ep w_s}ds>a\right)\ge \p\left( \left(e^{\ep w_{t_0}}\left(2\eta e^{-\delta\eta^\alpha}\right)>\frac{a}{\ep^2}\right)\cap {E_\delta}\right)\\
&\quad \quad \ge  \p\left( e^{\ep w_{t_0}}\left(2\eta e^{-\delta\eta^\alpha}\right)>\frac{a}{\ep^2}\right)- \p\left({E_\delta}\right).
\end{align*}
Since $\p\left({E_\delta}\right)=\p\left(\|\ep w\|_{\alpha}>\delta\right)=e^{-\frac{C_\delta}{\ep^2}}$ for some constant $C_\delta>0$, and
\begin{align*}
&\p\left( e^{\ep w_{t_0}}>\frac{ae^{\delta\eta^\alpha}}{2\eta\ep^2}\right)
=\p\left(\max_{t\in[0,1]} w_t>\frac{\log\left(\frac{ae^{\delta\eta^\alpha}}{2\eta}\right)}{\ep}+\frac{\log(1/\ep^2)}{\ep}\right)\\
&\quad\quad \approx \p\left(\max_{t\in [0,1]} w_t>\frac{\log(1/\ep^2)}{\ep}\right)\approx e^{-\frac{\log^2(1/\ep^2)}{2\ep^2}},
\end{align*}
hence we obtain the desired conclusion.
\end{proof}
Such a grading structure  is also reflected in the explicit calculations of transition density of $x(t)$ in Yor-Matsumoto \cite{MY} (also see Barrieu-Rouault-Yor \cite{BRY} and Gerhold \cite{Gerhold}).

At last we revisit the discussion of the non-local optimal path and related rate function $I^\ep$ for the event $(x^\ep(1)\in B_2)$, 
\[
\inf\left(I^\ep(\psi), \psi(1)\in B_2\right)=\inf\left(\frac{1}{2}\|h\|_{\h}^2, \ep^2\int_0^1 e^{\ep h_s}ds>a\right).
\]
It amounts to solve a variational problem, i.e. find the extremal of $\inf\left(\frac{1}{2}\|h\|_{\h}^2,\ep^2\int_0^1 e^{\ep h_s}ds=a\right)$.
 \begin{proposition}
For any $a>0$, we have
\begin{equation}\label{eq-solv-I-ep}
\inf\left(\frac{1}{2}\|h\|_{\h}^2,\ep^2\int_0^1 e^{\ep h_s}ds=a\right)=\frac{2\beta}{\ep^2}\left(\beta-\tanh \beta \right),
\end{equation}
where $\beta$ is the solution of $\frac{a}{\ep^2}=\frac{\sinh 2\beta}{2\beta}$.
The minimum is achieved at 
\[h_t=\frac1\ep\log \frac{\cosh^2\beta}{\cosh^2(\beta(1-t))}, \quad t\in[0,1]
\] and the optimal path is given by
\begin{equation}\label{eq-sol-optimal-path}
\Psi^\ep_{0}(h)_t=\bigg(\log \frac{\cosh^2\beta}{\cosh^2(\beta(1-t))}, \frac{\ep^2\cosh^2\beta}{\beta}(\tanh^2\beta-\tanh^2(\beta(1-t)))  \bigg).
\end{equation}
\end{proposition} 
\begin{proof}
Note 
\[
\inf\left(\frac{1}{2}\|h\|_{\h}^2,\ep^2\int_0^1 e^{\ep h_s}ds=a\right)=\frac1{\ep^2}\inf\left(\frac{1}{2}\|g\|_{\h}^2,\int_0^1 e^{g_s}ds=\frac{a}{\ep^2}\right)
\]
where $g_s=\ep h_s$. We use the method of Lagrange multiplier. Let
\[
\Lambda(g)=\frac12\int_0^1\dot{g}_s^2ds+\lambda\int_0^1e^{g_s}ds.
\]
Then $d\Lambda(g)\circ k=\int_0^1\dot{g}_s\dot{k}_sds+\lambda \int_0^1e^{g_s}k_sds=0$ for any $k\in \h$ implies that
\[
\dot{g}_1=0,\quad \ddot{g}_s=\lambda e^{g_s}.
\]
We can solve the above ODE explicitly and obtain that
\begin{equation}\label{eq-exp-h-t}
e^{g_t}=\frac{e^{g_1}}{\cosh^2\left(\sqrt{-\frac{\lambda}{2}}e^{\frac{g_1}{2}}(1-t)\right)}.
\end{equation}
Plug in the constraints $\int_0^1e^{g_s}ds=\frac{a}{\ep^2}$ and $g_0=0$ we then obtain
\begin{equation}\label{eq-beta-solv}
\frac{a}{\ep^2}=\frac{\sinh 2\beta}{2\beta},
\end{equation}
where $\beta=\sqrt{-\frac{\lambda}{2}}e^{\frac{g_1}{2}}$.
From \eqref{eq-exp-h-t} we can obtain that
\[
\frac12\|g\|^2_{\h}=2\beta\left(\beta-\tanh \beta \right).
\]
From \eqref{eq-exp-h-t} we can obtain the extremal 
\[
g_t=\cosh^2\beta-2\log \cosh(\beta(1-t)).
\]
Equation \eqref{eq-sol-optimal-path} then follows by plugging in $\Psi^\ep_{0}(h)_t=\left(\frac{g_t}{\ep}, \ep^2\int_0^te^{g_s}ds\right)$.
\end{proof}
\begin{Remark}
In particular, from \eqref{eq-beta-solv} we know that $\beta= \log \frac{1}{\ep}+o( \log \frac{1}{\ep})$. 
Plug it into \eqref{eq-solv-I-ep} we have
\[
\inf\left(\frac{1}{2}\|h\|_{\h}^2,\ep^2\int_0^1 e^{\ep h_s}ds=a\right)=\frac{2(\log\frac{1}{\ep})^2}{\ep^2}+o\left( \frac{(\log\frac{1}{\ep})^2}{\ep^2}\right).
\]
By Proposition \ref{lemma-affine} we then obtain that
\[
\lim_{\ep\to0} \frac{\log Q^{\ep}(B_2)}{\inf (I^{\ep}(\psi), \psi(1) \in B_2)} =-1,
\]
which agrees with the  statement in \eqref{claim1}.
\end{Remark}

\textbf{Acknowledgements:} The authors would like to thank S. R. S. Varadhan for his helpful advice on the example of ``bad" set in Section \ref{sec-bad-set}.
\clearpage


%
%
%
%
%

\end{document}